\documentclass{amsart}
\usepackage{geometry}                
\usepackage{graphicx}
\usepackage{amssymb}
\usepackage{epstopdf}
\usepackage[normalem]{ulem}
\usepackage{mathtools}
\usepackage{hyperref}

\usepackage[usenames,dvipsnames]{xcolor}
\usepackage{amsfonts,amssymb,amsmath,amsthm}
\usepackage{enumitem,mathrsfs}
\usepackage{mathtools}

\theoremstyle{plain}
\newtheorem{theorem}{Theorem}[section]
\newtheorem*{theorem*}{Theorem}
\newtheorem{lemma}[theorem]{Lemma}
\newtheorem{corollary}[theorem]{Corollary}
\newtheorem{proposition}[theorem]{Proposition}
\newtheorem*{main}{Main Theorem}

\theoremstyle{definition}
\newtheorem{definition}[theorem]{Definition}
\newtheorem{example}[theorem]{Example}

\theoremstyle{remark}
\newtheorem{remark}[theorem]{Remark}

\font\myfont=cmr10 at 14pt
\newcommand{\e}{{\text{\myfont e}}}

\newcommand{\MP}{\ensuremath{{\mathcal P}}}
\newcommand{\MZ}{\ensuremath{{\mathcal Z}}}
\newcommand{\ZZ}{\ensuremath{{\mathbb Z}}}
\newcommand{\CC}{\ensuremath{{\mathbb C}}}
\newcommand{\CW}{\ensuremath{{\widehat{\mathbb C}}}}
\newcommand{\RR}{\ensuremath{{\mathbb R}}}
\newcommand{\NN}{\ensuremath{{\mathbb N}}}

\newcommand{\E}{\ensuremath{{\mathscr E}}}
\newcommand{\R}{\ensuremath{{\mathcal R}}}
\newcommand{\ES}{\ensuremath{{\mathcal E}}}

\renewcommand{\Re}[1]{{\mathfrak{Re}\left(#1\right)}}
\renewcommand{\Im}[1]{{\mathfrak{Im}\left(#1\right)}}

\newcommand{\del}[2]{\frac{\partial #1}{\partial #2}}

\begin{document}

\title[Symmetries of vector fields with an essential singularity]
{Symmetries of complex analytic vector fields with an essential singularity
on the Riemann sphere}
\author[Alvarez--Parrilla \& Muci\~no--Raymundo]
{Alvaro Alvarez--Parrilla}
\author{Jes\'us Muci\~no--Raymundo}

\address{Grupo Alximia SA de CV, M\'exico} 
\email{alvaro.uabc@gmail.com}
\address{Centro de Ciencias Matem\'aticas, Universidad Nacional Aut\'onoma de M\'exico, M\'exico} 
\email{muciray@matmor.unam.mx}
 
\begin{abstract}
We consider the family

\centerline{
$\E(s,r,d)=\Big\{ X(z)=\frac{Q(z)}{P(z)}\ \e^{E(z)}\del{}{z} \Big\},$
}

\noindent
with $Q, P, E$ polynomials, $\deg{Q}=s$, $\deg{P}=r$ and $\deg{E}=d$,
of singular 
complex analytic vector fields $X$ on the Riemann sphere $\CW$. 
For $d\geq1$, $X\in\E(s,r,d)$ has $s$ zeros and $r$ poles on the complex plane 
and an essential singularity at infinity.
Using the pullback action of the affine group $Aut(\CC)$ and the 
divisors for $X$, 
we calculate the isotropy groups $Aut(\CC)_{X}$ 
and the discrete symmetries
for $X\in\E(s,r,d)$.
Each subfamily
$\E(s,r,d)_{id}$, of those  
$X$ with trivial isotropy group in $Aut(\CC)$,
is endowed with a holomorphic trivial principal $Aut(\CC)$--bundle structure.
Necessary and sufficient conditions in order to ensure 
the equality
$\E(s,r,d)=\E(s,r,d)_{id}$
and those $X\in\E(s,r,d)$ with non--trivial isotropy are realized.
Explicit global normal forms for $X\in\E(s,r,d)$ are presented.
A natural dictionary between vector fields, 
1--forms, quadratic differentials and functions is extended to include 
the presence of non--trivial discrete symmetries $\Gamma<Aut(\CC)$.
\end{abstract}

\maketitle
\section{Introduction}\label{Intro}

Meromorphic vector fields on compact Riemann surfaces are well understood, 
at least on some aspects: 
see  \cite{MR}, \cite{MucinoValero}, \cite{Branner-Dias}, 
\cite{Frias-Mucino}, \cite{Rousseau}.
Essential singula\-rities represent the next level of complexity.

We study the holomorphic 
families consisting of 
singular complex analytic vector fields 
on the Riemann sphere $\CW$
with a singular set composed of $s\geq0$ zeros 
and 
$r\geq0$ poles on $\CC$, and 
an isolated essential singularity at $\infty\in\CW$ of 1--order $d\geq 1$
namely
\begin{equation*}
\E(s,r,d)=
\left\{ X(z)=\frac{Q(z)}{P(z)}\ \e^{E(z)}\del{}{z} \ \Big\vert \ 
\begin{array}{l}
Q,\ P,\ E\in\CC[z],  \\
\deg{Q}=s,\ \deg{P}=r,\ \deg{E}=d 
\end{array}
\right\}.
\end{equation*} 

\noindent 
The associated families of functions 
$$\left\{ 
\Psi_X (z) = \int^z \frac{P(\zeta)}{Q(\zeta)}\ \e^{-E(\zeta)} d\zeta
\ \Big\vert \
X \in \E(s,r,d)
\right\}
$$ 
and their
Riemann surfaces $\{ \R_X \}$  
are part of the transcendental functions described 
in R. Nevanlinna's seminal work;
see \cite{Nevanlinna2} and 
\cite{Nevanlinna1} particularly ch.\ XI. 
Recently, M. Taniguchi studied these families $\{\Psi_X\}$
from the viewpoint of deformation of functions, 
see \cite{Taniguchi1} and \cite{Taniguchi3}. 
Motivated by complex dynamics,
K. Biswas and R. P\'erez--Marco, 
\cite{BiswasPerezMarco1}, \cite{BiswasPerezMarco2},
enrich the study of $\{ \Psi_X\}$ and $\{\R_X\}$.
In \cite{AlvarezMucino} the authors 
explored 
the family of vector fields  
$\E(0,0,d)$, obtaining an analytic classification 
as well as presenting analytic normal forms
for $d\leq3$. 

\smallskip
The search for a natural/adequate notion of normal form for vector 
fields in $\E(s,r,d)$ 
leads to novel paths.
A characteristic of the study of vector fields on the Riemann
sphere, or on the affine plane, is 
that their group of automorphisms is a finite dimensional complex analytic Lie group:
rich enough and yet treatable.
For $d\geq1$ the essential singularity of $X\in\E(s,r,d)$ 
provides a marked point at $\infty\in\CW$. 
We consider the canonical action 

\centerline{ 
$\mathcal{A}:Aut(\CC) \times \E(s,r,d)
\longrightarrow \E(s,r,d),
\ \ \ (T, X) \longmapsto T^* X$,} 

\noindent 
of the affine transformation group
$Aut(\CC)$ corresponding to those
$T \in Aut(\CW)= PSL(2, \CC)$ that fix $\infty$.

\smallskip 

Our pourpose is
the study of the 
quotient spaces $\E(s,r,d)/Aut(\CC)$. 

\noindent
Clearly it is a valuable and accurate tool for understanding the dynamics
of the vector fields $X \in  \E(s,r,d)$
and their associated 
families of functions $\{ \Psi_X \, \vert \, X \in  \E(s,r,d) \}$.

\smallskip

The action $\mathcal{A}$ 
determines the following natural classification problems:

\begin{enumerate}[leftmargin=.85cm]
\item[AC)]\label{A} 
Characterize under which conditions 
$X_1$ and $X_2$ in $\E(s,r,d)$
are \emph{complex analytically equivalent}, 
{\it i.e.} whether there exist $T\in Aut(\CC)$ such that

\centerline{
$X_{2} \xrightarrow{\ T^*} X_{1}$.}

\smallskip
\item[MC)]\label{B} 
Considering the singular flat metric $(\CW, g_X)$ associated to $X$,
characterize under which conditions 
the metrics associated to 
$X_1$ and $X_2$ in $\E(s,r,d)$ are
\emph{isometrically equivalent}; 
{\it i.e.} whether there exist $(T, \e^{i \theta}) 
\in Aut(\CC)\times \mathbb{S}^{1}$ such that

\centerline{
$(\CW ,g_{X_{2}}) \xrightarrow{\ T^*}
(\CW ,g_{\e^{i \theta}X_{1}})$, }

\noindent 
is an isometry, where 
$\e^{i \theta} : X \mapsto  \e^{i \theta } X$ acts by rotations. 
For the description of the metrics   
see \cite{MucinoValero}, \cite{MR}, \cite{AlvarezMucino}  and 
the \emph{singular complex analytic  dictionary} 
Proposition \ref{basic-correspondence}. 

\end{enumerate}

\noindent
The relation between (AC) and (MC), see Lemma \ref{Isometricamente}
for further detail, 
determines the following diagram

\begin{picture}(200,65)(-60,-15)

\put(10,35){$\E(s,r,d)$}
\put(60,35){\vector(1,0){30}}
\put(70,39){\small$\pi_1$}
\put(95,35){$\frac{\E(s,r,d)}{Aut(\CC)}$}
\put(150,35){\vector(1,0){30}}
\put(160,39){\small$\pi_2$}
\put(200,35){$\frac{\E(s,r,d)}{Aut(\CC)\times\mathbb{S}^1}$}

\put(110,27){\vector(0,-1){15}}
\put(110,12){\vector(0,1){15}}
\put(114,17){$\doteq$}
\put(220,27){\vector(0,-1){15}}
\put(220,12){\vector(0,1){15}}
\put(224,17){$\doteq$}

\put(82,-5){$\left\{ \begin{array}{c} normal \\ forms \ \lbrack X\rbrack \end{array} \right\}$}
\put(180,-5){$\left\{ \begin{array}{c} classes \ of \ flat \\ metrics \ (\CW,g_{X}) \end{array} \right\} ,$}

\put(-88,15){\vbox{\begin{equation}\label{diagrama-conclusion}\end{equation}}}
\end{picture}

\noindent
where $\pi_1$, $\pi_2$ are the natural projections to equivalence classes.

\noindent
As a first step to
enlighten both classifications, 
we study the $Aut(\CC)$--fibre bundle structure on $\E(s,r,d)$. 
Let 

\centerline{
$\E(s,r,d)_{id}\subseteq\E(s,r,d)$}

\noindent
denote those $X$ with 
\emph{trivial isotropy group $Aut(\CC)_{X}\subset Aut(\CC)$}.

\begin{main}
[Analytical and metric classification of $\E(s,r,d)$]\label{thmNormalForms}
\hfill
\begin{enumerate}[label=\arabic*),leftmargin=*]
\item The families
$\E(s,r,d)$ and $\E(s,r,d)_{id}$ coincide if and only if 
\\
$\bullet\ \gcd(d,s-r-1)=1$, or 
\\
$\bullet\ {\tt k}\hspace{-4pt}\not\vert s$ and ${\tt k}\hspace{-4pt}\not\vert r$, 
for all non--trivial common divisors ${\tt k}$ of $d$ and $(s-r-1)$.

\smallskip

\item For $s+r+d\geq 2$ and $d\geq1$,
the holomorphic (resp. real analytic) principal bundles

\begin{picture}(200,65)(-20,-15)
\put(15,35){$Aut(\CC)\longrightarrow \E(s,r,d)_{id}$}
\put(172,35){$Aut(\CC)\times \mathbb{S}^{1}\longrightarrow \E(s,r,d)_{id}$}

\put(91,30){\vector(0,-1){20}}
\put(95,20){$\pi_1$}
\put(270,30){\vector(0,-1){20}}
\put(274,20){$\pi_2\circ\pi_1$}

\put(72,-5){$\frac{\E(s,r,d)_{id}}{Aut(\CC)}$}
\put(250,-5){$\frac{\E(s,r,d)_{id}}{Aut(\CC)\times \mathbb{S}^{1}}$}

\put(140,0){,}
\put(-45,15){\vbox{\begin{equation}\label{diagramahaces}\end{equation}}}
\end{picture}

\noindent
are trivial. Moreover

\noindent 
$\bullet \ \E(s,r,d)_{id}/Aut(\CC)$ has 
complex dimension $s+r+d-1$, 

\noindent
$\bullet \ \E(s,r,d)_{id}/(Aut(\CC)\times \mathbb{S}^1)$ has real 
dimension $2(s+r+d)-3$ and

\noindent
both quotients are  compact when $\E(s,r,d) =\E(s,r,d)_{id}$.

\end{enumerate}
\end{main}

A natural tool for the study of $X$ is the \emph{divisor}
$$ 
\underbrace{[q_1, \ldots , q_s]}_{\MZ },  
\underbrace{[p_1, \ldots , p_r]}_{\MP },  
\underbrace{[e_1, \ldots , e_d]}_{\ES },
$$

\noindent
consisting of the roots of $Q(z)$, 
$P(z)$ and $E(z)$, 
see Definition \ref{divisor}. 
Some remarkable and novel features of $\E(s,r,d)$ are that 

\noindent $\bullet$
$\left(\MZ \cup\, \MP\right) \cap\, \ES$ 
need not be empty, and

\noindent $\bullet$
$\ES$ is not part of the singular set of the phase portrait of $\Re{X}$.

\smallskip
In order to show that 
$\E(s,r,d)_{id}$ is a 
holomorphic \emph{trivial} principal $Aut(\CC)$--bundle, in 
Lemma \ref{TrivialBundleLemma},
we exhibit explicit global sections

\centerline{
$\sigma: \E(s,r,d)_{id} / Aut(\CC)  \longrightarrow  \E(s,r,d)_{id}$.
}

\smallskip
We further localize the singular locus of 
the quotient $\E(s,r,d)/Aut(\CC)$, leading to a natural question:

\smallskip
\noindent 
\emph{
``How can we construct complex analytic
vector fields $X \in \E(s,r,d)$ such that 
$\Gamma$ concides with the symmetries $Aut(\CC)_X$ 
or is a proper subgroup of it?''
}

\begin{theorem*}[$\Gamma$--symmetry]
Let $	\Gamma$ be a non--trivial subgroup of $Aut(\CC)$.
A vector field $X\in\E(s,r,d)$ is $\Gamma$--symmetric if and only if 
$\Gamma$ is a discrete rotation group and 
\begin{enumerate}[label=\arabic*),leftmargin=*]
\item $\gcd(d,s-r-1)\neq 1$,
\item all three subsets of the divisor 
$[q_{1},\ldots,q_{s}]$, $[p_{1},\ldots,p_{r}]$, $[e_{1},\ldots,e_{d}]$
of $X$ are $\Gamma$--invariant.
\end{enumerate}
\end{theorem*}
\noindent
This result can be found restated as 
Theorem \ref{GammaInvariant-mejor} in the text, 
where an additional equivalent characterization is given.
It is clear that condition (2) is necessary, however it 
comes as a (pleasant) surprise that 
condition (1) provides sufficiency; 
compare with the case of $\Gamma$--symmetric rational functions \cite{Doyle-McMullen} \S~5 and $\Gamma$--symmetric rational vector fields \cite{AlvarezFriasYee}.

\smallskip

The explicit global sections found in Lemma \ref{TrivialBundleLemma} provide 
\emph{global normal forms} 
for vector fields 
$X\in\E(s,r,d)_{id}$, see   
Definition \ref{defNormalForm} and
Corollary \ref{teoNormalForm}.
The normal forms are global in the 
sense that the explicit expressions for $\sigma([X])$ are valid:

\noindent $\bullet$
for the whole family $\E(s,r,d)_{id}/Aut(\CC)$, and

\noindent $\bullet$
on the whole Riemann sphere $\CW$, when considering the 
phase portraits of $\Re{X}$.

\noindent 
Furthermore an application of Theorem \ref{GammaInvariant-mejor} allows us to
realize those $X\in\E(s,r,d)$ with non--trivial isotropy, thus providing normal forms
for all $X\in\E(s,r,d)$.

\smallskip

The above considerations lead to the following question.

\noindent
\emph{
``What is the relationship/link between vector fields and functions,
specifically between the families $\E(s,r,d)$
and $\{ \Psi_X \, \vert \, X \in  \E(s,r,d) \}$?''
}

\noindent
To answer this, consider an arbitrary Riemann surface $M$ (not necessarily compact).
In accordance with
\cite{KMS}, \cite{MucinoValero}, \cite{MR} and \cite{AlvarezMucino}, 
we present 
a Dictionary explaining the naturality and the richness of the theory:
a statement in one context can be restated in any other.

\theoremstyle{plain}
\newtheorem*{coro1}{Corollary} 
\begin{theorem*}[The dictionary under $\Gamma$--symmetry]  
Let $\Gamma$ be a subgroup 
of $Aut(M)$ having quotient 
$proj:M\longrightarrow M/\Gamma$
to a Riemann surface. 

\noindent
On $M$ there is a canonical one to one 
correspondence between:
\begin{enumerate}[label=\arabic*),leftmargin=*]
\item 
$\Gamma$--symmetric singular complex analytic vector fields $X$.
\item 
$\Gamma$--symmetric singular complex analytic differential forms $\omega_{X}$, 
satisfying $\omega_{X}(X)\equiv 1$.
\item 
$\Gamma$--symmetric singular complex analytic orientable quadratic differentials 
$\omega_{X} \otimes\omega_{X}$.
\item 
$\Gamma$--symmetric singular flat metrics 
$(M, g_{X})$    
with suitable singularities.
\item 
$\Gamma$--symmetric global singular complex analytic (possibly multivalued)
distinguished parameters $\Psi_X$.
\item 
Pairs $\big(\R_{X},\pi^{*}_{X,2}(\del{}{t})\big)$ 
consisting of branched Riemann surfaces $\R_{X}$, 
associated to the $\Gamma$--symmetric maps $\Psi_{X}$. 
\end{enumerate}
\end{theorem*}

A more complete statement  
is provided as Theorem \ref{RelacionCampoFuncion} and the calculation of
the singularites of $Y=proj_* X$, for $X \in \E(s,r,d)$ is performed
in Proposition \ref{propCociente} and Table \ref{Tabla}.

The groups of symmetries $\Gamma$ of Riemann surfaces and their
$\Gamma$--symmetric holomorphic tensors  
have been the subject of study in different works
from their own perspective.  
F. Klein was a pioneer \cite{Klein}, 
for more recent work see
\cite{Accola},
\cite{Farkas-Kra} ch. V,
\cite{Broughton} 
for the general theory of
automorphisms $Aut(M)$ and spaces of differentials,
\cite{Doyle-McMullen} \S5 for 
invariant rational functions on $\CW$;
and references therein.

In our case, 
examples of $\Gamma$--symmetric vector fields of the following kinds: rational, 
vector fields $X\in\E(s,r,d)$,
and vector fields on the torus $\mathbb{T}$ are presented. 
See Figures \ref{figE700}, \ref{fig7ejemploscampos}, and \ref{campoWeirstrass},
respectively.

\noindent
In Proposition \ref{propCociente} and 
Remark \ref{stratification}, 
the geometrical meaning of the 
subgroups $\Gamma<Aut(\CC)$ that leave invariant  
$X\in\E(s,r,d)$ is studied
by considering  
the natural projection 
$proj:\CW\longrightarrow\CW/ \Gamma$ and the associated vector fields $proj_* X$ 
on $\CW/ \Gamma$.

\smallskip 
The authors wish to thank 
Adolfo Guillot for useful comments.

\section{$Aut(\CC)$--fibre bundle structure on $\E(s,r,d)$}\label{compelxmanifoldstruct}

We work in the singular complex analytic category.  
Recalling definition 2.1 of \cite{AlvarezMucino}, 
for our present purposes; 
a \emph{singular analytic vector field on $\CW_{z}$}
is a holomorphic vector field $X$ on $\CW_{z}\backslash Sing(X)$,
with \emph{singular set}
$Sing(X)$ 
consisting of: 
{\it zeros} denoted by $\MZ$; 
{\it poles} denoted by $\MP$; 
{\it isolated essential singularity} at $\infty\in\CW$.

Because of Picard's theorem, even the local description of essential singularities of functions 
leads to a global study, see for instance \cite{AlvarezMucino} pp. 129.
Due to the diversity and wildness of essential singularities, 
a first step in understanding them is to restrict ourselves to the 
tame family $\E(s,r,d)$.

This section is devoted to the proof of the Main Theorem:

\noindent
In \S\ref{coordinates-for-E(s,r,d)} we provide explicit coordinates for $\E(s,r,d)$ that facilitate
the work to be done. 
In \S\ref{theAction} we present the action of $Aut(\CC)$ on $\E(s,r,d)$ and prove that 
$\E(s,r,d)_{id}$ is a trivial principal $Aut(\CC)$--bundle. 
Finally in \S\ref{obstrucciones} the arithmetic condition
\emph{``${\tt k}\hspace{-4pt}\not\vert q$ and 
${\tt k}\hspace{-4pt}\not\vert r$, 
for all non--trivial common divisors ${\tt k}$ of $d$ and $(s-r-1)$ 
implies that $\E(s,r,d)=\E(s,r,d)_{id}$''} is addressed. 
\subsection{Coordinates for $\E(s,r,d)$}
\label{coordinates-for-E(s,r,d)}

Vi\`ete's map provides a parametrization
of the space of monic polynomials of degree $s\geq1$
by the roots $\{q_{i}\}_{i=1}^{s}$, up to the action of the symmetric group of order $s$, 
$Sym(s)$. 
By parametrization we understand an atlas with appropriate charts; 
for instance in the case of the parametrization by roots, this is valid for 
neighborhoods that avoid multiple roots.
It will be useful to allow non--monic polynomials in the description of $X\in\E(s,r,d)$, explicitly
\begin{align*}
Q(z)=\lambda\, (z- q_1) \cdots (z- q_s) &:= \lambda\, (z^s + a_1 z^{s-1}+ \cdots + a_s),\\
P(z)=(z- p_1) \cdots (z- p_r) &:= z^r + b_1 z^{r-1}+ \cdots + b_r ,\\
E(z)=c_0\, (z- e_1) \cdots (z- e_d) &:= c_0\, \big(z^d + (c_1/c_0) z^{d-1}+ \cdots + (c_d/ c_0)\big).
\end{align*}

\begin{definition}\label{divisor}
The \emph{divisor of $X  \in \E(s,r,d)$} is 
$$ 
\underbrace{[q_1, \ldots , q_s]}_{\MZ },  
\underbrace{[p_1, \ldots , p_r]}_{\MP },  
\underbrace{[e_1, \ldots , e_d]}_{\ES },
$$

\noindent 
the unordered configuration of the roots of $Q(z)$, $P(z)$ and $E(z)$.
\end{definition}

Obviously we assume  
$\MZ \cap \MP = \varnothing$, however, 
$\left(\MZ \cup\, \MP\right) \cap\, \ES$,
need not be empty.
Different versions of the moduli space of $n$ 
points on the Riemann sphere under the action of $SL(2,\CC)$ are 
currently considered 
in the literature by using Mumford's geometric invariant theory GIT, 
see for instance \cite{Dolgachev}, \cite{Vakil} and references therein. 
In our case we consider $s+r+d$ unordered points with three ``flavors''.

\noindent 
The naturality of the divisors should come as no surprise:
in fact for $X\in \E(s,r,d)$ there is an identification 
between the zero dimensional object (the divisor) 
and the one dimensional object (the singular analytic vector field), 
see \cite{Tyurin} 
for other examples of the same phenomena.

\begin{proposition}\label{CartasParaErd}
The complex manifold
$\E(s,r,d)$ can be parametrized by:
\begin{enumerate}[label=\arabic*),leftmargin=*]
\item The $s+r+d+2$ coefficients 
\\ \centerline{
$\{ (\lambda, c_0, a_1,\ldots,a_s, b_1\ldots,b_r, c_1,\ldots,c_d) \} \subset (\CC^{*})^{2}\times\CC_{coef}^{s+r+d}$ 
}
of the polynomials $Q(z)$, $P(z)$ and $E(z)$.

\item The divisor of $X$ 
and the coefficients $\lambda$, $c_0$.
\end{enumerate}
\end{proposition}
\begin{proof}
Vi\`ete's map provides the first part of the diagram

\begin{picture}(200,70)(5,0)
\put(-3,55){$(\CC^*)^2 \times \left( \frac{\CC^s_{roots} }{Sym(s)} \right)
\times \left( \frac{\CC^r_{roots} }{Sym(r)} \right)
\times \left( \frac{\CC^d_{roots} }{Sym(d)} \right)
\rightarrow
(\CC^*)^{2} \times \mathbb{C}^{s+r+d}_{coef}
\rightarrow
(\CC^*)^{2} \times \mathbb{C}^{s+r+d}_{coef}
$}

\put(3,37){$
(\lambda, c_0,  [ q_1, \ldots , q_s ], [p_1, \ldots , p_r ] , [e_1,\ldots, e_d ])
\mapsto
$}
\put(104,22){$
(\lambda, c_0, a_1, \ldots , a_s, b_1, \ldots, b_r, c_1, \ldots, c_d )
\mapsto
$}
\put(160,8){
$\lambda\,\frac{ z^{s}+a_1z^{s-1}+\ldots+a_s }{z^{r}+b_1 z^{r-1}+\ldots+b_r }
\exp(c_0 z^d + \ldots  + c_d)\, \del{}{z}$.}

\put(-10,15){\vbox{\begin{equation}\label{Vieta}\end{equation}}}
\end{picture}
\begin{remark}
\label{redundance}
The parameters $\lambda$, $c_0$ and $c_d$ are interrelated.

\noindent 
In fact, when writing $X\in\E(s,r,d)$ in terms of the roots, 
both $\lambda$ and $c_0$ are needed:
$c_d$ does not appear explicitly in the description, but the roots 
$\lbrack e_1, \ldots, e_d\rbrack$ depend on 
both $c_0$ and $c_d$. 

\noindent
On the other hand, when writing $X\in\E(s,r,d)$ in terms of the 
coefficients, either $\lambda$ or $c_d$ is redundant, but $c_0$ is indispensable.

\noindent
This redundancy/interrelationship has  virtues as will be seen 
in \S\ref{theAction}.
\end{remark}

\noindent
To be precise, equation \eqref{Vieta} with $\lambda=1$, 
provides complex analytic charts in a fundamental domain for 
the action of 
$Sym(s)\times Sym(r)\times Sym(d)$ on 
$\CC^{*}\times\CC_{roots}^{s}\times\CC_{roots}^{r}\times\CC_{roots}^{d}$.
\end{proof}

\subsection{The action of $Aut(\CC)$ on $\E(s,r,d)$}
\label{theAction}

The group $Aut(\CC)$ of
complex automorphisms determines 
the complex analytic equivalence (AC) and the 
isometric equivalence (MC) 
for $\E(s,r,d)$ as in the Introduction.

\begin{lemma}\label{Isometricamente}
If two vector fields $X_1,\ X_2\in\E(s,r,d)$ are 
analytically equivalent on $\CC$, then the 
associated singular flat metrics $g_{X_1}$ and $g_{X_2}$ are orientation preserving isometrically equivalent.

\noindent
Conversely, if $g_{X_1}$ and $g_{X_2}$ are orientation preserving isometrically equivalent, then necessarily 
$\e^{i\theta}X_1=T^* X_2$, 
for $(T,\e^{i\theta})\in Aut(\CC)\times \mathbb{S}^1$.
\end{lemma}
\begin{proof}
Use the ideas for the equivalence between vector fields $X$ and singular flat metrics with a 
unitary geodesic foliation as in \cite{AlvarezMucino} pp.~137.
\end{proof}
\ 

\vspace{-10pt}
Compare the dimension of $Aut(\CC)$ to the case of a the group of smooth 
automorphisms of the sphere, $\text{\it Diff\,}^\infty(\mathbb{S}^2)$, which is infinite dimensional; 
or to the case of a compact Riemann surface 
$M_g$ of genus $g\geq 2$ that has 
finite automorphism group, see \cite{Farkas-Kra} ch. V.
The case $g=1$ does admit a large automorphism group for $M_g$, however, 
in this work we only consider the Riemann sphere.

Denote the 
\emph{stabilizer} or \emph{isotropy group} of $X\in\E(s,r,d)$ by 
$$Aut(\CC)_{X}=\{T\in Aut(\CC) \ \vert\ T^*X=X\}.$$
We shall say that $\Gamma < Aut(\CC)$ \emph{leaves invariant} $X\in\E(s,r,d)$ 
if $\Gamma$ is a 
subgroup of $Aut(\CC)_{X}$.
Of course this is equivalent to saying that \emph{$X$ is $\Gamma$--symmetric}.

\noindent
Further, let
$$\E(s,r,d)_{id}=\big\{X\in\E(s,r,d)\ \vert \ Aut(\CC)_{X}=\{id\} \big\},$$
be the family consisting of those $X$ with trivial isotropy. 
It is immediate 
that $\E(s,r,d)_{id}$ is open and dense in $\E(s,r,d)$. 
Finding necessary and sufficient conditions in order to ensure the equality 
is a central question.

Recalling Proposition \ref{CartasParaErd}, 
a virtue of the root parametrization \eqref{Vieta} and the parameter 
$\lambda$, is as follows.
The action of $Aut(\CC)=\{ T: w \mapsto {\tt a}w + {\tt b}=z \}$
by pullback is 
\begin{multline*}
\begin{array}{ccl}
\mathcal{A}:
Aut(\CC) \times\ \E(s,r,d) & \longrightarrow & \E(s,r,d)\\
\big( {\tt a}w+{\tt b}, 
(\lambda, c_0,  [ q_1, \ldots , q_s ], [p_1, \ldots , p_r ] , 
[e_1,\ldots,e_d ]) \big)
&\longmapsto&
\end{array}
\\
\Big({\lambda }{\tt a}^{s-(r+1)} ,{c_0 }{\tt a}^{d} ,
[T^{-1}(q_1), \ldots , T^{-1}(q_s)],
\\ 
[T^{-1}(p_1), \ldots , T^{-1}(p_r)], [T^{-1}(e_1), \ldots, T^{-1}(e_d)] \Big).
\end{multline*}
Explicitly,
\begin{multline}\label{completeaction}
T^* \left( 
\lambda\,\frac{(z-q_1)\cdots(z-q_s )}
{(z-p_1)\cdots(z-p_r )}
\exp\big(c_0 (z-e_1)\cdots (z-e_d)\big)\,
\del{}{z}
\right) 
\\
=\lambda\,\frac{{\tt a}^{s}} {{\tt a}^{r+1}} \frac{(w-T^{-1}(q_1)) \cdots (w-T^{-1}(q_s)) }
{(w-T^{-1}(p_1)) \cdots (w-T^{-1}(p_r))} 
\\
\exp\big(c_0 {\tt a}^d (w-T^{-1}(e_1)) \cdots (w-T^{-1}(e_d))\big)\,
\del{}{w} .
\end{multline}
\noindent
With this expression for the action we will be able to prove the following.
\begin{lemma}
\label{GammaInvariant}
Let $X\in\E(s,r,d)$, and consider the set 
\begin{equation}\label{condicionS}
\mathscr{D}=\left\{{\tt k}\in\NN\ \vert \ {\tt k} \text{ is a common divisor of }d \text{ and }(s-r-1) \right\}.
\end{equation}
A non--trivial subgroup $\Gamma < Aut(\CC)$ leaves invariant $X$ if and only if
\begin{enumerate}[label=\arabic*),leftmargin=*]

\item $\Gamma$ is a discrete rotation group, {\it i.e.}
$$\Gamma=\left\{T(w)=\e^{i2\pi j/{\tt k}}w+{\tt b}
 \ \vert \ j=1,\ldots,{\tt k} \right\}\cong\ZZ_{\tt k},$$
 \noindent 
for some ${\tt k}\in \mathscr{D} \backslash \{ 1\}$. 
The center of rotation of $\Gamma$ is  

\centerline{
$C\doteq{\tt b}/(1-\e^{i 2\pi/{\tt k}}) \in\CC$. }

\item 
All three subsets $\MZ$, $\MP$ and $\ES$, of the divisor of $X$,
are $\Gamma$--invariant, 
in particular each subset 
is evenly distributed on concentric circles about $C$. 
\end{enumerate}
\end{lemma}

Of course $Aut(\CC)_{X}$ is the biggest subgroup $\Gamma$ that leaves invariant  $X$, 
so we immediately have.

\begin{corollary}
\label{isotropy}
The isotropy group of $X\in\E(s,r,d)$ is non--trivial if and only if the
following conditions occur
\begin{enumerate}[label=\arabic*.\hspace{-3pt},leftmargin=*]

\item (Arithmetic condition)
$\mathscr{D} \backslash \{ 1\}\neq\varnothing$.

\item (Geometric condition)
All three subsets $\MZ$, $\MP$ and $\ES$, of the divisor of $X$,
are $Aut(\CC)_X$--invariant. 
\hfill\qed
\end{enumerate}
\end{corollary}

\begin{remark}\label{barycenters}
Recall Definition \ref{divisor},
the geometric condition (2) implies that
$C\in\CC$ coincides with 
the 

\centerline{\emph{barycenters} 
{\tt Z} of $\MZ$, {\tt P} of $\MP$ and {\tt E} of $\ES$. }

\noindent 
This is a necessary but not sufficient condition
in order to have non--trivial isotropy group.
\end{remark}

In order to gain some intuition, consider the following 
simple examples. 

\begin{example}
\label{ejemploE023}
Consider 

\centerline{$X(z)=-\frac{\e^{z^{3}}}{3z^{2}}\del{}{z}\in\E(0,2,3),$}

\noindent 
its divisor is 

\centerline{
$\MZ =\varnothing,\ \MP =[0, 0],\ \ES =[0, 0, 0]$}

\noindent 
which is clearly invariant by $\ZZ_{3}$. 
Moreover the common divisors of $d=3$ and 
$s-r-1=0-2-1=-3$ are $\mathscr{D}=\{1,3\}$.
Hence, by Corollary \ref{isotropy} it follows that the isotropy group of $X$ is $\ZZ_{3}$,
see Figure \ref{fig7ejemploscampos} (A). 
\end{example}

\begin{example}\label{ejemploE033s}
Consider

\centerline{$X(z)=\frac{\e^{z^{3}}}{3z^{3}-1}\del{}{z}\in\E(0,3,3),$}

\noindent 
its divisor is 

\centerline{$ \MZ= \varnothing,\  
\MP=[1/3,\e^{i2\pi/3}/3,\e^{-i2\pi/3}/3],\ \ES=[0, 0, 0]$}

\noindent 
which is clearly invariant by $\ZZ_{3}$. 
However the common divisors of 
$d=3$ and $s-r-1=0-3-1=-4$ are $\mathscr{D}=\{1\}$.
So, even though 
$X$ satisfies the geometric condition of Corollary \ref{isotropy}, 
it does not satisfy the arithmetic condition, which implies that 
its isotropy group is the identity.
See Figure \ref{fig7ejemploscampos} (B). 
\end{example}
\begin{figure}[htbp]
\begin{center}
\includegraphics[width=0.9\textwidth]{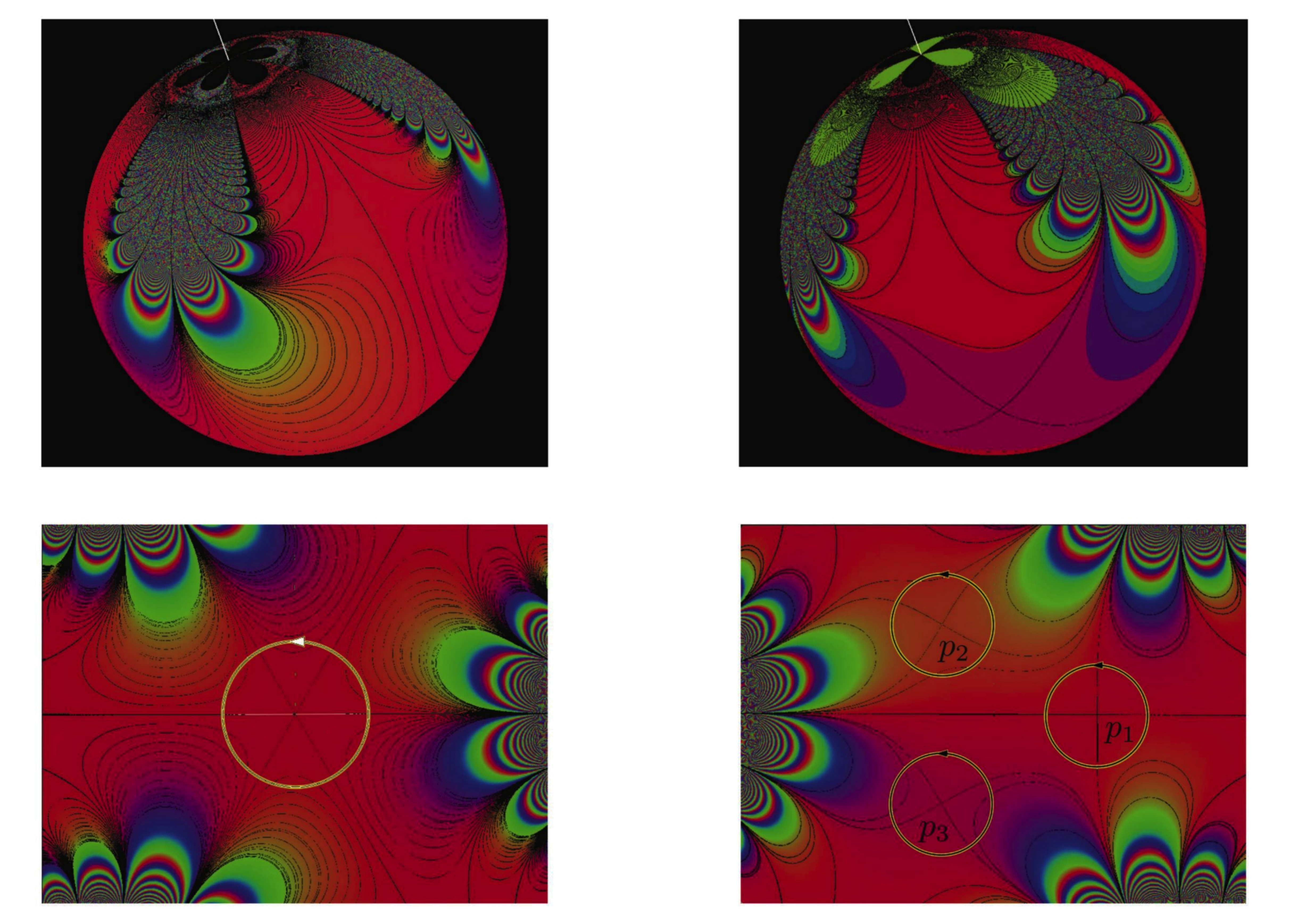}
\\
(A)\hspace{200pt} (B)
\caption{Phase portrait of  
Examples \ref{ejemploE023} and \ref{ejemploE033s}.
Borders of the strip flows correspond to streamlines of $\Re{X}$.
(A) The vector field 
$X(z)=-\frac{\e^{z^{3}}}{3z^{2}}\del{}{z}\in\E(0,2,3)$ 
with isotropy group isomorphic to $\ZZ_{3}$. 
(B) The vector field $X(z)=\frac{\e^{z^{3}}}{3z^{3}-1}\del{}{z}\in\E(0,3,3)$ 
with isotropy group the identity.
On the top we observe the projective view and on the bottom the affine view.
}
\label{fig7ejemploscampos}
\end{center}
\end{figure}

\begin{remark}
All the figures of vector fields  
were obtained using the visualization techniques 
presented in \cite{AlvarezMucinoSolorzaYee}. 
In particular, the streamlines of $\Re{X}$ are represented as the borders of the \emph{strip flows} (represented as bands of the same color) or, in particular cases that need to be emphasised, as individual trajectories. 
See \S6.2 of the same reference for further explanation of the numerical behaviour
at zeros, poles and essential singularities.
\end{remark}


\begin{proof}[Proof of Lemma \ref{GammaInvariant}]
Let $X\in\E(s,r,d)$ be a singular complex analytic vector field.
It follows immediately, from \eqref{completeaction}, 
that $T^{*}X=X$ for some non--trivial $T \in Aut(\CC)$
if and only if 
\begin{enumerate}[label=\alph*),leftmargin=*]
\item ${\tt a}^{d}={\tt a}^{s-r-1}=1$, and
\item all three sets $\MZ$, $\MP$ and $\ES$ 
are $T$--symmetric.
\end{enumerate}

\smallskip 

Note that condition (a) is equivalent to ${\tt a}=\e^{i 2\pi/{\tt k}}$, with 
${\tt k}\in\mathscr{D}\backslash\{1\}$, $\mathscr{D}$
as in \eqref{condicionS}.
So 
$T(w)=\e^{i 2\pi j/{\tt k}}w + {\tt b}$
for $j=1,\ldots,{\tt k}$ and  ${\tt b}\in\CC$ as above.

Since all of $s,r,d<\infty$, condition (b) implies that $T$ can not be a parabolic transformation;
{\it i.e.} $T$ has two distinct fixed points in $\CW$. 
One of them is $\infty$, so if ${\tt b}\neq0$ then ${\tt k}\neq 1$, 
which in turn implies that $T$ is a non--trivial rotation with center $C$.
\end{proof}

In particular, if ${\tt k}=\gcd(d,s-r-1)=1$ then $Aut(\CC)_X=\{id\}$. 

\smallskip
As is usual the triviality of the isotropy group of $X\in\E(s,r,d)$ 
has geome\-tric implications on the quotient spaces. 

\begin{remark}
From the description \eqref{completeaction} of the action $\mathcal{A}$, 
of $Aut(\CC)$ on $\E(s,r,d)_{id}$
in terms of the divisor of $X$, it is clear that for 
$s+r+d \geq 2$,
$\mathcal{A}$ is a proper map.
\end{remark}

It is well known, see for instance 
\cite{DuistermaatKolk} pp.~53, that 
the quotient $\E(s,r,d)_{id}/Aut(\CC)$ is a manifold of dimension 
$dim(\E(s,r,d)_{id})-dim (Aut(\CC))$. 
Naturally $\E(s,r,d)_{id}$ is open and dense in $\E(s,r,d)$, thus 
$dim(\E(s,r,d)_{id})=dim(\E(s,r,d))$.
The analogous fact holds for the action of $Aut(\CC) \times \mathbb{S}^{1}$.

\noindent
From this it follows that both 
$$
\pi_1:\E(s,r,d)_{id}\longrightarrow\frac{\E(s,r,d)_{id}}{Aut(\CC)}
\text{ \ \ and \ \ }
(\pi_{2}\circ\pi_{1}):\E(s,r,d)_{id}\longrightarrow
\frac{\E(s,r,d)_{id}}{Aut(\CC)\times \mathbb{S}^{1}},
$$ 

\noindent
in \eqref{diagramahaces},
are holomorphic and real--analytic principal $Aut(\CC)$ 
and $(Aut(\CC)\times \mathbb{S}^{1})$--bundles, respectively.

\begin{lemma}\label{TrivialBundleLemma}
Let $s+r+d \geq 2$ 
and $d\geq1$,
then $\E(s,r,d)_{id}$ is a holomorphic trivial principal $Aut(\CC)$--bundle.
\end{lemma}
When $d=0$ the isotropy group $Aut(\CC)_{X}$ for $X\in\E(s,r,0)$ does not generically
fix $\infty\in\CW$, see \S\ref{familiasEspeciales} for further details.
\begin{proof}

\noindent
On $\E(s,r,d)_{id}$, every fiber is a copy of $Aut(\CC)$. 
We shall explicitly exhibit three choices of \emph{global} sections. 
We start by recalling that 
$X\in\E(s,r,d)$ can be expressed as
\begin{equation*}
X(z) = \lambda\ \frac{z^{s} + a_{1} z^{s-1} + a_{2} z^{s-2} + \ldots + a_{s}}
{z^{r} + b_{1} z^{r-1} + b_{2} z^{r-2} + \ldots + b_{r}} 
\exp\big(c_0 z^{d} + c_{1} z^{d-1} + c_2 z^{d-2} + \ldots + c_{d-1} z \big) \ \del{}{z},
\end{equation*}
since the coefficient $c_{d}$ can be incorporated in $\lambda$.

Next we consider a ``gauge transformation prospect'' 
\begin{equation*}
\begin{array}{ccl}
\mathcal{G}: \frac{\E(s,r,d)_{id}}{Aut(\CC)} & \longrightarrow & Aut(\CC)\\
{[X]} & \longmapsto & G(w)={\tt a} w + {\tt b},
\end{array}
\end{equation*}
with suitable ${\tt a}$ and ${\tt b}$ that will depend on the specific 
representative $X$ of the class $[X]$.
We shall now procede to choose appropiate ${\tt a}$ and ${\tt b}$. 

\noindent
$\bullet$ 
The choice ${\tt a}= \big(\frac{1}{c_0}\big)^{1/d}$, forces 
the polynomial that appears in the exponential of the expression for 
$(G^{*}X)(w)$ to be monic.

\noindent 
$\bullet$ 
Recalling that the barycenters of $\MZ$, $\MP$ and $\ES$ are
${\tt Z}=-a_{1}/s$, ${\tt P}=-b_{1}/r$ and ${\tt E}=-c_{1}/(d c_{0})$ respectively, 
we shall choose ${\tt b}$ such that one of the polynomials appearing in 
the expression for $(G^{*}X)(w)$ is centered.

\noindent
This provides us with the following three explicit 
\emph{global} sections:
\begin{enumerate}[label=\alph*),leftmargin=*]

\item $d\geq 2$: In this case, given $[X]\in\frac{\E(s,r,d)_{id}}{Aut(\CC)}$, 
choose 
${\tt b}=-\frac{c_1}{d c_0}={\tt E}$ (so $G^{-1}({\tt E})=0$); 
we then obtain the global section
\begin{equation}\label{seccionGlobalD2}
\begin{array}{ccl}
\sigma: \frac{\E(s,r,d)_{id}}{Aut(\CC)} & \longrightarrow & \E(s,r,d)_{id}\\
{[X]} & \longmapsto & (G^{*}X)(w)=\widetilde{\lambda}\ 
\frac{w^{s} + \widetilde{a}_1 w^{s-1} + \widetilde{a}_2 w^{s-2} + \ldots + \widetilde{a}_s}
{w^{r} + \widetilde{b}_1 w^{r-1} + \widetilde{b}_2 w^{r-2} + \ldots + \widetilde{b}_r} \hfill \\
& & \hspace{72pt} \exp\big(w^{d} + \widetilde{c}_2 w^{d-2} 
+ \ldots + \widetilde{c}_{d-1} w \big) \ \del{}{w}.
\end{array}
\end{equation}
\noindent
That is, all three polynomials are monic and the one appearing in the exponential 
of the expression for $(G^{*}X)(w)$ is centered.

\noindent
A special case is when $\MZ= \MP=\varnothing$ and $d\geq 2$,  

\centerline{
$(G^{*}X)(w)=\widetilde{\lambda} \exp(w^{d}
+\widetilde{c}_{2}w^{d-2}
+\ldots+\widetilde{c}_{d-1}w) \del{}{w}$.
}

\noindent
Compare with \S8.6 of \cite{AlvarezMucino}.

\item $s\geq 1$: In this case, given $[X]\in\frac{\E(s,r,d)_{id}}{Aut(\CC)}$, 
choose  
${\tt b}=-\frac{a_1}{s}={\tt Z}$ (so $G^{-1}({\tt Z})=0$); 
we then obtain the global section
\begin{equation}\label{seccionGlobalS2}
\begin{array}{ccl}
\sigma: \frac{\E(s,r,d)_{id}}{Aut(\CC)} & \longrightarrow & \E(s,r,d)_{id}\\
{[X]} & \longmapsto & (G^{*}X)(w)= \widetilde{\lambda}\
\frac{w^{s} + \widetilde{a}_2 w^{s-2} + \ldots + \widetilde{a}_s}
{w^{r} + \widetilde{b}_1 w^{r-1} + \ldots + \widetilde{b}_r} \hfill\\
& & \hspace{63pt} \exp\big(w^{d} 
+ \widetilde{c}_1 w^{d-1} + \ldots + \widetilde{c}_{d-1} w \big) \ \del{}{w}.
\end{array}
\end{equation}
\noindent
That is, all three polynomials are monic and the one corresponding to the zeros 
of $(G^{*}X)(w)$ is centered.

\item $r\geq 1$: In this case, given $[X]\in\frac{\E(s,r,d)_{id}}{Aut(\CC)}$, 
choose
${\tt b}=-\frac{b_1}{s}={\tt P}$ (so $G^{-1}({\tt P})=0$); 
we then obtain the global section
\begin{equation}\label{seccionGlobalR2}
\begin{array}{ccl}
\sigma: \frac{\E(s,r,d)_{id}}{Aut(\CC)} & \longrightarrow & \E(s,r,d)_{id}\\
{[X]} & \longmapsto & (G^{*}X)(w)= \widetilde{\lambda}\
\frac{w^{s} + \widetilde{a}_1 w^{s-1} + \ldots + \widetilde{a}_s}
{w^{r} + \widetilde{b}_2 w^{r-2} + \ldots + \widetilde{b}_r} \hfill\\
& & \hspace{63pt} \exp\big( w^{d} 
+ \widetilde{c}_1 w^{d-1} + \ldots + \widetilde{c}_{d-1} w \big) \ \del{}{w}.
\end{array}
\end{equation}
\noindent
That is, all three polynomials are monic and the one corresponding to the poles 
of $(G^{*}X)(w)$ is centered.

\end{enumerate}
\noindent
Finally, note that any $(s,r,d)$ such that $\E(s,r,d)_{id}$ is an $Aut(\CC)$--bundle falls in one 
of the above cases.
\end{proof}

\begin{remark}
Recalling that $\mathbb{S}^1$ acts by 
$\e^{i \theta}: X \longmapsto \e^{i \theta} X$ 
and preserves the 
singular flat metric $g_X$;
the normal forms given by  \eqref{seccionGlobalD2}, \eqref{seccionGlobalS2} 
and \eqref{seccionGlobalR2} can be extended to consider the 
action of $Aut(\CC)\times \mathbb{S}^{1}$  
by requiring that $\widetilde{\lambda}  \in \RR^+$.
This then produces the desired vector field $X$ and $(\CW, g_X)$ in normal form. 
\end{remark}

In order to finish the proof of the Main Theorem, we still need to show the arithmetic condition
\emph{``${\tt k}\hspace{-4pt}\not\vert q$ and 
${\tt k}\hspace{-4pt}\not\vert r$, 
for all non--trivial common divisors ${\tt k}$ of $d$ and $(s-r-1)$ 
implies that $\E(s,r,d)=\E(s,r,d)_{id}$''}.

\subsection{Obstructions for the existence of non--trivial symmetries}
\label{obstrucciones}

The purpose of this section is to characterize those vector fields $X\in\E(s,r,d)$ that 
have non--trivial isotropy group 
$Aut(\CC)_{X}\cong\ZZ_{\tt k}$, 
for ${\tt k}\in\mathscr{D}\backslash\{1\}$, recall \eqref{condicionS}.

\noindent
From Corollary \ref{isotropy} we see that there are two obstructions for the existence of 
$X\in\E(s,r,d)$ with $Aut(\CC)_{X}\neq \{id\}$.

\noindent
With this in mind we shall start by considering the partition of $\MZ$, $\MP$ and $\ES$ into orbits under the 
action of $Aut(\CC)_{X}$. 
\begin{remark}[Orbit structure]\label{PQorbits}
Recalling that $C$ is the fixed point of the discrete rotation group $\Gamma$, it is evident that:

\noindent
\emph{The configurations $\MZ$, $\MP$ and $\ES$ 
are $Aut(\CC)_{X}$--symmetric if and only if 
each configuration $\MZ$, $\MP$ and $\ES$ is evenly distributed on circles 
(of any given radius $R\geq0$) 
centered about the fixed point $C$, 
generically on more than one circle}.

\noindent
Moreover, as will be shown, $C\in\MZ\cup\MP$.
\end{remark}
From \eqref{completeaction}, it is clear that the set of poles and zeros of $X$ do not intersect, 
that is $\MZ\cap\MP=\varnothing$; however $\ES$ is 
unrelated to $\MZ$ and $\MP$, in the sense that $\ES\cap\MZ$ and $\ES\cap\MP$ may be non--empty.

\noindent
Assuming $\gcd(d,s-r-1)\neq 1$ let $X\in\E(s,r,d)\backslash\E(s,r,d)_{id}$.

\noindent
The search for an alternative for Lemma \ref{GammaInvariant} is expressed as (A), (B) and (C) below.
\begin{enumerate}[label=\Alph*),leftmargin=*]
\item\label{choice-of-s} Choose ${\tt k}\in\mathscr{D}\backslash\{1\}$ and let it remain fixed.

\item\label{placement-of-roots} For the $d$ roots $\ES$ of the polynomial $E(z)$, 
recall the orbit structure of Remark \ref{PQorbits} and 
proceed as follows:
\begin{enumerate}[label=\roman*)]
	\item Consider the partitions of $d$ as a sum of positive integers, say 
	\begin{equation*}
	Part(d)=\left\{ \{d_{\iota,\kappa}\}_{\iota=1}^{\ell_\kappa} \Big\vert\  
	d=\sum\limits_{\iota=1}^{\ell_{\kappa}}d_{\iota,\kappa}, \kappa=1,\ldots,p(d) \right\},
	\end{equation*}
	where $p(d)$ is the partition function of $d$ (the number of possible integer partitions of $d$). 
	
	\item\label{s-en-circulo} 
	Let $\{d_{\iota,\kappa}\}_{\iota=1}^{\ell_\kappa}$ be a partition such that 
	$d_{j,\kappa}= {\tt k} \nu_j$, for some $\nu_j\in\NN$, say 
	\begin{equation*}
	d=d_{1,\kappa}+d_{2,\kappa}+\ldots
	+\underbrace{d_{j,\kappa}}_{\ \, = {\tt k} \nu_j}
	+\ldots+d_{\ell_{\kappa},\kappa},
	\end{equation*} 
	choose this partition and place ${\tt k}$ equally spaced roots on a circle $L_{j}$ centered about $C$ of a 
	chosen radius $R_{j}>0$, all with the same multiplicity $\nu_j$.
	
	\item\label{s-otro-circulo} If there are still some $d_{\iota,\kappa}={\tt k} \nu_\iota$, for $\nu_\iota\in\NN$ 
	in the same partition,  
	place ${\tt k}$ equally spaced roots on a circle $L_{i}$ centered about $C$ 
	(possibly the same circle as before 
	but the roots are to be placed on different positions), 
	once again each root with multiplicity $\nu_\iota$.
	Repeat (\ref{s-otro-circulo} if possible or proceed to (\ref{resto-en-origen} below.
	
	\item\label{resto-en-origen} Finally, place the rest of the roots at $C$; hence $C$ will be a root of $E(z)$
	of multiplicity equal to $d$ minus the number of roots (counted with multiplicity) already placed on 
	circles of positive radius.
	\end{enumerate}
\item\label{placement-of-pole-zeros} For the placement of the poles and zeros of $X$, 
we proceed as in (B) replacing ``$d$'' and ``roots of $E(z)$'' with ``$r$'' and ``roots of $P(z)$'', and ``$s$'' and 
``roots of $Q(z)$'', respectively. 
\\
However since ${\tt k}\vert(s-r-1)$, then ${\tt k}\vert s$ and ${\tt k}\vert r$ can not occur simultaneously;
leaving the following cases.
\begin{enumerate}[label=\alph*)]
	\item\label{no-dividen} ${\tt k}\hspace{-4pt}\not\vert s$ and ${\tt k}\hspace{-4pt}\not\vert r$.
	\item\label{polo-centro} ${\tt k}\vert s$ and ${\tt k}\hspace{-4pt}\not\vert r$. 
	\item\label{cero-centro} ${\tt k}\hspace{-4pt}\not\vert s$ and ${\tt k}\vert r$.
\end{enumerate}
Case (\ref{no-dividen} can not occur: if ${\tt k}\hspace{-4pt}\not\vert s$ then we must place a zero of $X$ 
at the fixed point $C$ of the rotation
(by considering the partitions of $s$ as a sum of positive integers as in 
(\ref{placement-of-roots}, 
it follows from the orbit structure, {\it i.e.} Remark \ref{PQorbits}, that at least one zero of $X$ 
must be placed at $C$).
Similarly if ${\tt k}\hspace{-4pt}\not\vert r$ then we must place a pole of $X$ at the fixed point $C$ of the rotation; 
but $\MZ\cap\MP=\varnothing$. 

\noindent
Case (\ref{polo-centro} requires a pole of $X$ at the fixed point $C$ and case (\ref{cero-centro} requires a 
zero of $X$ at the fixed point $C$. Thus either (\ref{polo-centro} or (\ref{cero-centro} occurs, but not both.
\end{enumerate}

The arithmetic conditions stated as cases (b) and (c) above can be interpreted geometrically 
as \emph{$C$ has to be either a pole or a zero of $X$}, respectively. 
However, since $X$ have non--trivial isotropy group, 
then there are
local restrictions on the allowed multiplicity $\nu$ of $C$.

\noindent
Consider the phase portrait in a neighborhood of the center of rotation $C\in\CC$.
This together with the fact that the non--trivial isotropy groups are the discrete rotation 
groups $\ZZ_{\tt k}$ with 
${\tt k}\in\mathscr{D}\backslash\{1\}$, implies that:

\begin{enumerate}[label=\alph*)]
\item When $C$ is a pole of $X$ of multiplicity $\nu$, the phase portrait of $X$ in a 
neighborhood of $C$ 
consists of $2(\nu+1)$ hyperbolic sectors. Since hyperbolic sectors come in pairs, 
${\tt k}\vert (\nu+1)$ is 
required.
\item On the other hand, when $C$ is a zero of $X$ of multiplicity $\nu$, the phase portrait of $X$ in a 
neighborhood of $C$ consists of $2(\nu-1)$ elliptic sectors. Since elliptic sectors come in pairs, 
${\tt k}\vert (\nu-1)$ is required.
\end{enumerate}

With this in mind we can now restate Lemma \ref{GammaInvariant}.
\begin{theorem}\label{GammaInvariant-mejor}
Let $X\in\E(s,r,d)$.
The discrete rotation group

\centerline{
$\Gamma=\left\{T(w)=\e^{i2\pi j/{\tt k}}w+{\tt b}, j=1,\ldots,{\tt k} \right\}$
$\cong\ZZ_{\tt k}$, ${\tt k}\geq2$, ${\tt b}\in\CC$
}

\noindent
leaves invariant $X$
if and only if
\begin{enumerate}[label=\arabic*),leftmargin=*]
\item ${\tt k}$ is a common divisor of $d$ and $(s-r-1)$, 

\item either 
\begin{enumerate}[label=\alph*),leftmargin=*]
\item {\bf (${\tt k}\vert s$ and ${\tt k}\hspace{-5pt}\not\hspace{-2pt}\vert r$): $C$ is a pole of $X$} of multiplicity
$\nu\geq1$ with ${\tt k}\vert(\nu+1)$; 
furthermore the rest of the poles and all the zeros are evenly distributed on circles centered about $C$, thus $r={\tt k} k_r + \nu$ with ${\tt k}\hspace{-4pt}\not\vert\nu$ and $s={\tt k} k_s$,

\hspace{-22pt} or

\item {\bf (${\tt k}\hspace{-4pt}\not\hspace{-2pt}\vert s$ and ${\tt k}\vert r$): $C$ is a zero of $X$} of multiplicity 
$\nu\geq1$ with ${\tt k}\vert(\nu-1)$;
furthermore the rest of the zeros and all the poles are evenly distributed on circles centered about $C$, thus $s={\tt k} k_s + \nu$ with ${\tt k}\hspace{-4pt}\not\vert\nu$ and $r={\tt k} k_r$,
\end{enumerate}

\item $\ES$ is evenly distributed on circles centered about $C$, thus $d={\tt k} k_d$.
\end{enumerate}
Otherwise $Aut(\CC)_X=\{id\}$.
\end{theorem}

\begin{proof}
Condition (1) is a restatement of (1) of Lemma \ref{GammaInvariant}.
The discussion previous to the statement of Theorem \ref{GammaInvariant-mejor} together 
with \eqref{completeaction} are enough
to show that conditions (2) and (3) are equivalent to (2)
of Lemma \ref{GammaInvariant}.
\end{proof}

\begin{remark}
1. Theorem \ref{GammaInvariant-mejor} provides a way of realizing those $X\in\E(s,r,d)$ 
that are $\Gamma$--symmetric for $\Gamma\cong\ZZ_{\tt k}$, 
${\tt k}\in\mathscr{D}\backslash\{1\}$. 
See \S\ref{realizationIsotropy} for the explicit construction.

\noindent
2. Note that the divisibility conditions on the multiplicity $\nu$ of the pole or zero at the fixed point $C$ 
are automatically satisfied. 

\noindent
That is, if (1), (3) and (${\tt k}\vert s$ and ${\tt k}\hspace{-4pt}\not\vert r$) are satisfied, then 
$r={\tt k} k_r + \nu$ for some $\nu\geq1$ with ${\tt k}\hspace{-4pt}\not\vert\nu$ and ${\tt k}\vert(\nu+1)$.

\noindent
Similarly, if (1), (3) and (${\tt k}\hspace{-4pt}\not\vert s$ and ${\tt k}\vert r$) are satisfied, then 
$s={\tt k} k_s + \nu$ for some $\nu\geq1$ with ${\tt k}\hspace{-4pt}\not\vert\nu$ and ${\tt k}\vert(\nu-1)$.

\noindent
Both statements follow from \eqref{completeaction}.
\end{remark}

\noindent
As an immediate consequence of Theorem \ref{GammaInvariant-mejor} we have:
\begin{corollary}\label{coroIsotropia}
$Aut(\CC)_{X}=\{id\}$ if and only if 
\\
$\bullet\ \gcd(d,s-r-1)=1$, or 
\\
$\bullet\ {\tt k}\hspace{-4pt}\not\vert s$ and ${\tt k}\hspace{-4pt}\not\vert r$, 
for all non--trivial common divisors ${\tt k}$ of $d$ and $(s-r-1)$.
\hfill\qed
\end{corollary}

\noindent
Which in turn finishes the proof of the Main Theorem.

\smallskip
Note that as stated in (1) of the Main Theorem, even if $\gcd(d,s-r-1)\neq 1$ it is possible that 
$\E(s,r,d)=\E(s,r,d)_{id}$, as the next example shows.
\begin{example}[$\gcd(d,s-r-1)\neq1$ does not guarantee the existence of $X$ with non--trivial symmetry]
Let $s=11$, $r=7$ and $d=6$, then $\gcd(d,s-r-1)=\gcd(6,3)=3\neq 1$.
However $3\hspace{-4pt}\not\vert 11$ and $3\hspace{-4pt}\not\vert 7$. 
Thus by Corollary \ref{coroIsotropia}, $\E(11,7,6)=\E(11,7,6)_{id}$.
\end{example}

On the other hand for $\E(s,r,d)=\E(s,r,d)_{id}$ we must check that the condition 
``${\tt k}\hspace{-4pt}\not\vert s$ and ${\tt k}\hspace{-4pt}\not\vert r$'', is satisfied \emph{for all}
non--trivial common divisors ${\tt k}$ of $d$ and $(s-r-1)$.
\begin{example}[Not all common divisors of $d$ and $s-r-1$ give rise to symmetry]
Let $s=35$, $r=4$ and $d=30$, then $\gcd(d,s-r-1)=\gcd(30,30)=30\neq 1$. 
Moreover $\mathscr{D}=\{1,2,3,5,6,10,15,30\}$ and we see that 
\begin{align*}
2\vert 35 &\text{ and } 2\hspace{-4pt}\not\vert 4
&
3\hspace{-4pt}\not\vert 35 &\text{ and } 3\hspace{-4pt}\not\vert 4
&
5\hspace{-4pt}\not\vert 35 &\text{ and } 5\vert 4
\\
6\hspace{-4pt}\not\vert 35 &\text{ and } 6\hspace{-4pt}\not\vert 4
&
10\hspace{-4pt}\not\vert 35 &\text{ and } 10\hspace{-4pt}\not\vert 4
&
15\hspace{-4pt}\not\vert 35 &\text{ and } 15\hspace{-4pt}\not\vert 4
\\
& & 30\hspace{-4pt}\not\vert 35 &\text{ and } 30\hspace{-4pt}\not\vert 4.
\end{align*}
It follows, from Theorem \ref{GammaInvariant-mejor}, that only $\ZZ_{\it k}$ with ${\it k}=2, 5$
can be non--trivial symmetry groups for $X\in\E(35,4,30)$. 
In fact
$$X_{2}(z)=\frac{z^{35}}{z^4-1}\e^{z^{30}} \del{}{z} , 
\quad 
X_{5}(z)=\frac{(z^5-1)^7}{z^4}\e^{z^{30}} \del{}{z} 
\in\E(35,4,30)$$
are $\ZZ_2$--invariant and $\ZZ_5$--invariant, respectively.
So $\E(35,4,30)\neq\E(35,4,30)_{id}$. 
\end{example}

We point out some relevant particular cases.
\begin{remark}
\label{casosespeciales} 
1. The special case $\E(0,0,d)=\E(0,0,d)_{id}$ since $s=r=0$ so $\gcd(d,-1)=1$.
See also theorem 8.16 in \cite{AlvarezMucino}. 

\noindent
2. For each $d\geq2$ there are $X\in\E(0,d-1,d)$ such that $Aut(\CC)_{X}=\ZZ_{d}$. 
Thus in fact all the cyclic groups appear as isotropy groups of $X\in\E(0,r,d)$ for appropriate pairs $(r,d)$. 

\noindent
3. For each sufficiently large pair $(r,d)$ with $\gcd(d,r+1)\neq1$, there are an 
infinite number of non conformally 
equivalent configurations of the roots $\ES$ of $E(z)$ and $\MP$ of $P(z)$ 
which are invariant by the non--trivial $T\in Aut(\CC)_{X}\neq\{id\}$. 
This follows from Remark \ref{PQorbits} and the fact that 
the quotient of the radii of an annulus is a conformal invariant; 
thus there are an infinite number of possible configurations of the roots $\ES$ and $\MP$. 
\end{remark}

\noindent
This last special case can be re--stated as: 
\begin{corollary}
For each sufficiently large pair $(r,d)$ with ${\tt k} =\gcd(d,r+1)\neq1$, there are an infinite number of 
non conformally equivalent $X\in\E(0,r,d)$ with isotropy group $Aut(\CC)_{X}\cong\ZZ_{\tt k} $.
\end{corollary}

\section{Normal forms for $\E(s,r,d)$}\label{FormasNormales} 
We start with a formal definition. 

\begin{definition}\label{defNormalForm}
A \emph{normal form} of $X \in \E(s,r,d)$ is a 
representative of its class 
under the pullback action $\mathcal{A}$ of $Aut(\CC)$.
\end{definition}
\noindent
The explicitness of 
the global sections, Lemma \ref{TrivialBundleLemma}, immediately provides us with.

\begin{corollary}[Normal forms for $\E(s,r,d)_{id}$]\label{teoNormalForm}
For $s+r+d\geq 2$ 
and $d\geq1$, 
global normal forms for $X\in\E(s,r,d)_{id}$ are given by $(G^* X)(w)$ 
as in \eqref{seccionGlobalD2}, \eqref{seccionGlobalS2} and \eqref{seccionGlobalR2}. 
\hfill\qed
\end{corollary} 

\begin{remark}
The term \emph{global} refers to the fact that the expressions for $(G^* X)(w)$ 
given by \eqref{seccionGlobalD2}, \eqref{seccionGlobalS2} and \eqref{seccionGlobalR2}
are valid for every $X\in\E(s,r,d)_{id}$ and also throughout $\CW$.
\end{remark}

Furthermore, an application of Theorem \ref{GammaInvariant-mejor} enables 
us to also find the normal forms for $X\in\E(s,r,d)$ with non--trivial isotropy. 

\subsection{Realizing $X\in\E(s,r,d)$ with non--trivial isotropy group}
\label{realizationIsotropy}
We proceed as follows:

\noindent
1) $d$ and $(s-r-1)$ must have non--trivial divisors ${\tt k}\in\mathscr{D}\backslash\{1\}$.

\noindent
Given $\Gamma\cong\ZZ_{\tt k}$ a discrete rotation group,  
$X\in\E(s,r,d)$ is $\Gamma$--symmetric if and only if the following two conditions occur:

\noindent
2) The configuration of poles $\MP$ and zeros $\MZ$ of $X$ are 
$\Gamma$--symmetric
and either
\vspace{-10pt}
\begin{enumerate}[label=\alph*)]
	\item $X$ has a pole as a fixed point of $\Gamma$, of multiplicity $\nu$ with ${\tt k}\vert(\nu+1)$, or
	

	\item $X$ has a zero as a fixed point of $\Gamma$, of multiplicity $\nu$ with ${\tt k}\vert(\nu-1)$.
\end{enumerate}
\vspace{-10pt}
3) The configuration $\ES$ of roots of $E(z)$ are $\Gamma$--symmetric.

\subsubsection{Zeros and poles with arbitrary multiplicity} 
With the above in mind we immediately obtain.

\begin{theorem}[Realizing vector fields with non--trivial symmetry]\label{realization}
Consider $\E(s,r,d)$ with $\mathscr{D}\backslash\{1\}\neq\varnothing$. 
%
The discrete rotation group

\centerline{
$\Gamma=\left\{T(w)=\e^{i2\pi j/{\tt k}}w+{\tt b}, j=1,\ldots,{\tt k} \right\}$
$\cong\ZZ_{\tt k}$, ${\tt k}\in\mathscr{D}\backslash\{1\}$, ${\tt b}\in\CC$,
}

\noindent
with center of rotation

\centerline{
$C\doteq{\tt b}/(1-\e^{i 2\pi/{\tt k}}) \in\CC$, }

\noindent
leaves invariant those $X\in\E(s,r,d)$ that satisfy the following conditions.

\begin{enumerate}[label=\arabic*),leftmargin=*]
\item (${\tt k}\vert s$ and ${\tt k}\hspace{-4pt}\not\vert r$): in this case $C$ is a pole,
furthermore
\begin{equation*}
X(z) = \lambda\,\frac{ \prod\limits_{j=1}^{k_{s}} \prod\limits_{\ell=1}^{\tt k} 
\Big[z-C - (r_{j}\e^{i\theta_j})^{\ell/{\tt k}} \Big] }
{ (z-C)^{\nu} \prod\limits_{j=1}^{k_{r}} \prod\limits_{\ell=1}^{\tt k}
\Big[z-C - (R_{j}\e^{i\alpha_j})^{\ell/{\tt k}} \Big]} 
\exp\left\{ (z-C)^{\mu} \prod\limits_{j=1}^{k_{d}}\prod\limits_{\ell=1}^{\tt k} 
\Big[z-C - (\rho_{j}\e^{i\beta_j})^{\ell/{\tt k}} \Big] \right\} \del{}{z},
\end{equation*}
for choices of ${\tt k},k_{s},k_{r},k_{d}$ such that $s={\tt k} k_{s}$, $r={\tt k} k_{r}+\nu$, $d={\tt k} k_{d}+\mu$,
$\{r_j\},\{R_j\},\{\rho_j\} \subset\RR^{+}$, $\{\theta_j\},\{\alpha_j\},\{\beta_j\} \subset\RR$, $\mu\in\NN\cup\{0\}$ 
and $\nu\in\NN$ such that ${\tt k}\vert(\nu+1)$.

\item (${\tt k}\hspace{-4pt}\not\vert s$ and ${\tt k}\vert r$): in this case $C$ is a zero, 
furthermore
\begin{equation*}
X(z) = \lambda\,\frac{ (z-C)^{\nu} \prod\limits_{j=1}^{k_{s}} \prod\limits_{\ell=1}^{\tt k} 
\Big[z-C - (r_{j}\e^{i\theta_j})^{\ell/{\tt k}} \Big] }
{ \prod\limits_{j=1}^{k_{r}} \prod\limits_{\ell=1}^{\tt k}
\Big[z-C - (R_{j}\e^{i\alpha_j})^{\ell/{\tt k}} \Big]} 
\exp\left\{ (z-C)^{\mu} \prod\limits_{j=1}^{k_{d}}\prod\limits_{\ell=1}^{\tt k} 
\Big[z-C - (\rho_{j}\e^{i\beta_j})^{\ell/s} \Big] \right\} \del{}{z},
\end{equation*}
for choices of ${\tt k},k_{s},k_{r},k_{d}$ such that $s={\tt k} k_{s}+\nu$, $r={\tt k} k_{r}$, $d={\tt k} k_{d}+\mu$,
$\{r_j\},\{R_j\},\{\rho_j\} \subset\RR^{+}$, $\{\theta_j\},\{\alpha_j\},\{\beta_j\} \subset\RR$, $\mu\in\NN\cup\{0\}$ 
and $\nu\in\NN$ such that ${\tt k}\vert(\nu-1)$.
\hfill
\qed
\end{enumerate}
\end{theorem}
\begin{remark}\label{FormaNormalNoTrivial}
Note that the expressions in Theorem \ref{realization} are in fact normal forms for  
$X\in\E(s,r,d)\backslash\E(s,r,d)_{id}$.
\end{remark}

\subsubsection{Simple zeros and simple poles in $\CC$}
The case of $X$ having simple poles and simple zeros has further structure. 
Let 

\centerline{
$\E(s,r,d)^{S}:=\{X\in\E(s,r,d) \ \vert\ \text{all the poles and zeros of }X\text{ in }\CC
\text{ are simple} \}$.}

\noindent
From the orbit structure (Remark \ref{PQorbits}), Theorem \ref{GammaInvariant-mejor} 
and the fact that only simple poles and zeros 
are allowed, it follows that $s={\tt k} k_{s}+1$ or $r={\tt k} k_{r}+1$, 
with $k_{s},k_{r}\in\NN\cup\{0\}$.

Let us first consider the case $r={\tt k} k_{r}+1$ with $k_{r}\geq0$. 
Then if we want $X\in\E(s,r,d)^{S}$ to have non--trivial isotropy group, 
we must require that $s={\tt k} k_{s}$, 
with $k_{s}\geq0$. 
On the other hand ${\tt k}\vert(s-r-1)=((k_{s}-k_{r}){\tt k} -2)$ hence ${\tt k} =2$ and $d=2 k_{d}$ 
for $k_{d}\geq1$.
We have then proved.
\begin{proposition}
\label{casosimplepolofijo}
Let $X\in\E(s,r,d)^{S}$ have non--trivial isotropy group fixing a pole of $X$. 
Then $r=2 k_{r}+1$ for $k_r\in\NN\cup\{0\}$, 
$Aut(\CC)_{X} \cong \ZZ_{2}$ and the vector fields $X$ are of the form
\begin{equation*}
X(z) = \lambda\,\frac{ \prod\limits_{j=1}^{k_{s}} \Big[(z-C)^2 - q_{j}^2\Big] }{(z-C) 
\prod\limits_{j=1}^{k_{r}} 
\Big[(z-C)^2 - p_{j}^2\Big]} 
\exp\left\{ (z-C)^{2\mu} \prod\limits_{j=1}^{k'_{d}} \Big[(z-C)^2 - e_{j}^2\Big] \right\} \del{}{z},
\end{equation*}
where $k_{r}=\frac{r-1}{2}\geq0$, $k_{s}=\frac{s}{2}\geq0$, $\mu\geq0$, $k'_{d}=\frac{d-2\mu}{2}\geq0$, all the $\{p_{j}\}\subset\CC\backslash\{0\}$ and $\{q_{j}\}\subset\CC\backslash\{0\}$ are distinct, and the $\{e_{j}\}\subset\CC$ need not be distinct.\hfill\qed
\end{proposition}

\begin{example}
\label{ejemploE032s}
Let 
$$X(z)=\frac{\e^{z^{2}}}{z(z^2+1)} \del{}{z}\in\E(0,3,2)^{S}.$$
Its isotropy group is $Aut(\CC)_{X}=\ZZ_{2}$, see Figure \ref{E03E05} (c).
\end{example}

However the case $s={\tt k} k_{s}+1$ with $k_{s}\geq0$ is different. 
In this case, upon a similar examination we have.
\begin{proposition}
\label{casosimplecerofijo}
For each ${\tt k} \geq2$, let $s= {\tt k} k_s +1\geq1$, $r= {\tt k} k_r \geq0$ and 
$d={\tt k} k_d \geq1$ 
for $k_s, k_r, k_d\in\NN\cup\{0\}$. 
Then there is an $X\in\E(s,r,d)^{S}$ with non--trivial isotropy group 
$Aut(\CC)_{X} \cong \ZZ_{s}$ fixing a zero of $X$. 

\noindent
These vector fields $X$ are of the form
\begin{equation*}
X(z) = \lambda\,\frac{ (z-C) \prod\limits_{j=1}^{k_{s}} \prod\limits_{\ell=1}^{\tt k} 
\Big[z-C - (r_{j}\e^{i\theta_j})^{\ell/{\tt k}} \Big] }
{ \prod\limits_{j=1}^{k_{r}} \prod\limits_{\ell=1}^{\tt k}
\Big[z-C - (R_{j}\e^{i\alpha_j})^{\ell/{\tt k}} \Big]} 
\exp\left\{ (z-C)^{\mu} \prod\limits_{j=1}^{k_{d}}\prod\limits_{\ell=1}^{\tt k} 
\Big[z-C - (\rho_{j}\e^{i\beta_j})^{\ell/{\tt k}} \Big] \right\} \del{}{z},
\end{equation*}
for choices of $\{r_j\},\{R_j\},\{\rho_j\} \subset\RR^{+}$, $\{\theta_j\},\{\alpha_j\},\{\beta_j\} \subset\RR$ and 
$\mu\in\NN\cup\{0\}$
such that 
$\{ (r_{j}\e^{i\theta_j})^{\ell/{\tt k}} \}$ and $\{ (R_{j}\e^{i\alpha_j})^{\ell/{\tt k}} \}$ are distinct, 
but the $\{ (\rho_{j}\e^{i\beta_j})^{\ell/{\tt k}} \}$ need not 
ne\-cessarily be distinct.\hfill\qed
\end{proposition}

\begin{remark}[Simple poles and zeros]
Proposition \ref{casosimplepolofijo} and Proposition \ref{casosimplecerofijo} can be summarized as:
\begin{enumerate}[label=\arabic*)]
\item If there is a (simple) pole of $X$ at the fixed point $C\in\CC$, then the number of (simple) zeros of $X$ is 
even, the number of (simple) poles of $X$ is odd and the number of roots (counted with multiplicity) of the 
polynomial in the exponential, is even.
\item If there is a (simple) zero of $X$ at the fixed point $C\in\CC$, there is no restriction other than those given 
by the orbit structure (Remark \ref{PQorbits}).
\end{enumerate}
\end{remark}

\section{Singular complex analytic dictionary and $\Gamma$--symmetry}\label{Appendix}

\subsection{The dictionary}
Previously, the authors presented a \emph{dictionary/correspondence} in the complex analytic framework, 
which is stated below as Proposition \ref{basic-correspondence}, 
in particular it applies 
to $X$ (and $\Psi_X$)
in the family $\E(s,r,d)$.
A complete proof can be found in \cite{AlvarezMucino} \S2.2 with further discussion in 
\cite{AlvarezMucinoSolorzaYee}.
\begin{proposition}[Singular complex analytic dictionary] 
\label{basic-correspondence}
\hfill\\
On any (non necessarily compact) Riemann surface $M$ 
there is a canonical one to one correspondence 
between:
\begin{enumerate}[label=\arabic*),leftmargin=*]
\item 
Singular complex analytic vector fields $X$.
\item 
Singular complex analytic differential forms $\omega_{X}$, satisfying $\omega_{X}(X)\equiv 1$.
\item 
Singular complex analytic orientable quadratic differentials $\omega_{X} \otimes\omega_{X} $.
\item 
Singular flat metrics $(M,g_{X})$ with suitable singularities, 
trivial holonomy and provided with a 
real geodesic vector field $\Re{X}$, arising from $\omega_{X} \otimes\omega_{X}$ satisfying $g_{X}(\Re{X},\Re{X})\equiv 1$ and $g_{X}(\Re{X},\Im{X})\equiv 0$.
\item 
Global singular complex analytic (possibly multivalued)
distinguished parameters
\\
\centerline{ 
$\Psi_{X} (z)= \int^z \omega_{X} : M\longrightarrow \CW_{t}.$
}
\item 
Pairs $\big(\R_{X},\pi^{*}_{X,2}(\del{}{t})\big)$ 
consisting of branched Riemann surfaces $\R_{X}$, 
associated to the maps $\Psi_{X}$, and
the vector fields $\pi^{*}_{X,2}(\del{}{t})$ under the projection 
$\pi_{X,2}: \R_{X} \longrightarrow \CW_{t}$.
\qed 
\end{enumerate}
\end{proposition}
\smallskip
\noindent
To better understand the dictionary, note that:
The singular set of $X$, $Sing(X)$, is composed of zeros,  
poles, essential singularities and accumulation points of the above. 
The adjectives ``singular complex analytic'' should be clear 
for each of the objects in Proposition \ref{basic-correspondence}.
The \emph{singular flat metric $g_{X}$ with singular set
$Sing(X)$} is the flat 
Riemannian metric on 
$M\backslash Sing(X)$ defined as the pullback under 

\centerline{
$\Psi_{X}:(M , g_{X})\rightarrow (\CC_{t},\vert dt \vert )$,}

\noindent 
where $\vert dt \vert$ 
is the usual flat Riemannina metric on $\CC_{t}$, 
see \cite{MucinoValero}, \cite{MR} and \cite{AlvarezMucino}. 
The topology of the phase portrait of $\Re{X}$
and the geometry of $g_X$ 
are subjects of current interest,
some pioneering sources can be found in 
\cite{AlvarezMucino} at 
\S1, pp.~133, 
\S5 pp.~159 and table 2. 
See \cite{AlvarezMucinoSolorzaYee} for visualizational aspects.
Applications of geometric structures associated to 
flat metrics $(\CW, g_X)$ can be found in  \cite{Guillot}.

The graph of $\Psi_{X}$
\begin{equation*}
\R_{X}= \{(z,t) \ \vert \  t=\Psi_{X}(z) \} \subset M\times\CW_{t}
\end{equation*}
\noindent 
is a Riemann surface provided with the vector field
induced by $\big(\CW,\del{}{t}\big)$ 
via the projection of $\pi_{X,2}$, 
say $\big(\R_{X},\pi_{X,2}^{*}(\del{}{t})\big)$.

\noindent 
Moreover the singular flat metric from this pair 
coincides with $g_{X}=\Psi_{X}^{*} \vert dt \vert $ 
since $\pi_{X,1}$ is an isometry 
(the isometry is to be understood on the complement of 
the corresponding singular 
set in $\R_{X}$). 
We summarize all this in the diagram
\begin{center}
\begin{picture}(180,70)(5,20)

\put(-126,40){\vbox{\begin{equation}\label{diagramaRX}\end{equation}}}

\put(12,75){$\big(M,X\big) $}

\put(115,75){$\big(\R_X,\pi^*_{X,2}(\del{}{t})\big)$}

\put(108,78){\vector(-1,0){60}}
\put(65,85){$\pi_{X,1}$}

\put(133,65){\vector(0,-1){30}}
\put(138,47){$ \pi_{X,2} $}

\put(38,65){\vector(2,-1){73}}
\put(55,39){$ \Psi_X $}

\put(115,20){$\big(\CW_t,\del{}{t}\big). $}

\end{picture}
\end{center}

In the presence of non--trivial symmetries we have.

\begin{theorem}[The dictionary under $\Gamma$--symmetry]  
\label{RelacionCampoFuncion}
Let $\Gamma$ be a subgroup
of the complex automorphisms $Aut(M)$ 
having quotient 
$proj:M\longrightarrow M/\Gamma$
to a Riemann surface.
\begin{enumerate}[label=\arabic*.,leftmargin=*]
\item On $M$ there is a canonical one to one 
correspondence between:
\begin{enumerate}[label=\arabic*),leftmargin=*]
\item 
$\Gamma$--symmetric singular complex analytic vector fields $X$.
\item 
$\Gamma$--symmetric singular complex analytic differential forms $\omega_{X}$, 
satisfying $\omega_{X}(X)\equiv 1$.
\item 
$\Gamma$--symmetric singular complex analytic orientable quadratic differentials 
$\omega_{X} \otimes\omega_{X}$.
\item 
$\Gamma$--symmetric singular flat metrics $(M,g_{X})$ with suitable singularities.
\item 
$\Gamma$--symmetric global singular complex analytic (possibly multivalued)
distinguished parameters $\Psi_X$.
\item 
Pairs $\big(\R_{X},\pi^{*}_{X,2}(\del{}{t})\big)$ 
consisting of branched Riemann surfaces $\R_{X}$, 
associated to the $\Gamma$--symmetric maps $\Psi_{X}$. 
\end{enumerate}

\smallskip 

\item 
Moreover, any $X$ (resp. $\Psi_X$) on $M$ which is invariant by a non--trivial
$\Gamma<Aut(M)$ can be recognized as a lifting of  
a suitable vector field $Y$ (resp. function $\Psi_{Y}$) on $M/\Gamma$,
as in the following diagram
\begin{center}
\begin{picture}(180,150)(30,20)

\put(-107,70){\vbox{\begin{equation}\label{diagramaEquivariante}\end{equation}}}


\put(-15,75){$\big(M/\Gamma,Y\big) $}

\put(90,75){$\big(\R_{Y},\pi^*_{Y,2}(\del{}{t})\big)$}

\put(88,78){\vector(-1,0){50}}
\put(55,68){$\pi_{Y,1}$}

\put(113,65){\vector(0,-1){30}}
\put(87,50){$ \pi_{Y,2} $}

\put(18,65){\vector(2,-1){73}}
\put(35,39){$ \Psi_{Y} $}

\put(95,20){$\big(\CW_t,\del{}{t}\big). $}

\put(75,155){$\big(M,X \big) $}

\put(180,155){$\big(\R_X,\pi^*_{X,2}(\del{}{t})\big)$}

\put(178,158){\vector(-1,0){65}}
\put(135,165){$\pi_{X,1}$}

\put(203,145){\vector(0,-1){30}}
\put(208,130){$ \pi_{X,2} $}

\put(108,145){\line(2,-1){50}}
\put(166,116){\vector(2,-1){17}}
\put(140,130){$ \Psi_X $}

\put(185,100){$\big(\CW_t,\del{}{t}\big) $}


\put(190,145){\vector(-1,-1){57}}
\put(126,110){$\widetilde{\text{proj}}_*$}

\put(80,145){\vector(-1,-1){57}}
\put(25,120){proj$_*$}

\put(190,93){\vector(-1,-1){60}}
\put(170,60){$id$}

\end{picture}
\end{center}
\end{enumerate}
\end{theorem}

\begin{proof}
By hypothesis
$ proj: M \longrightarrow M /\Gamma $
determines a connected Riemann surface as a target,
thus Diagram \eqref{diagramaRX} holds true
both for $M$ and $M/ \Gamma$. 

We want to show that 
$proj_* X \doteq Y$ is a well defined vector field on $M/\Gamma$.

\noindent 
From a local point of view,  let $(\CC, 0)$ denote local charts of $M$ 
where $0$ corresponds to a fixed point for some $g: M \longrightarrow M$,
$g \neq id$ in $\Gamma$.
Without loss of generality, we assume that $proj^{-1}(proj (\CC, 0))$
is connected in $M$. 

\noindent 
Note that $X$ is necessarily singular at $(\CC, 0)$. 
The trouble is that the local behaviour of $X$ is unknown. 
The computation of $Y$ from the germ $\big( (\CC, 0), X \big)$ is 
by using geometrical arguments. 
The fundamental domain of 

\centerline{$proj: (\CC, 0) \longmapsto (\CC, 0) / \Gamma$}

\noindent 
is an angular sector
$\{ 0 \leq  arg(z) \leq  2\pi/ \kappa \} \subset (\CC, 0)$, 
$\kappa \geq 2$. 
Using the singular flat metric $g_X$ and the frame of 
geodesic vector fields $\Re{X}, \Im{X}$ on the angular sectors
(recall Theorem \ref{basic-correspondence} (4)), 
the value of $X$ at  the borders of an angular sector coincide,
hence
the germ  $Y$ on $proj \big( (\CC, 0), Y \big)$
is well defined. 

\noindent 
For poles, zeros and the simplest exponential isolated
singularities 
at $(\CC, 0)$ explicit computations are
provided in Table \ref{Tabla}, 
which in itself is of independent interest.

\noindent 
The global existence of $Y$ on $M/\Gamma$
follows by an analytic continuation argument.

Diagram \eqref{diagramaEquivariante} for vector fields 
follows immediately, 
where $proj_*$ and $\widetilde{proj}_*$ are the maps induced 
by $proj$ on $M$ 
and $\R_{X}$ respectively. 

Finally, the use of the dictionary extends Diagram
\eqref{diagramaEquivariante} to singular complex analytic 
1--forms $\omega_X$ and functions $\Psi_X$;  
where $g \in \Gamma$ acts on functions as
$\Psi_X \mapsto \Psi_X \circ g $.  
Assertions (2) and (5) are done. 
\end{proof}
\ 
\vspace{-15pt}

As a matter of record, in Table \ref{Tabla} 
the linear vector field $\lambda z \del{}{z}$ has 
complete isotropy group
$\CC^*$;
however only discrete groups are considered for 
Theorem \ref{RelacionCampoFuncion}. 
However, Table \ref{Tabla} makes sense globally,
in the last row we use $(\CW, \infty)$ as germ domain.
\begin{table}[htp]
\caption{Computation of $Y=proj_* X$ given a germ $\big((\CC,0),X\big)$.}
\begin{center}
\begin{tabular}{|c|c|c|c|c|c|}
\hline
\multicolumn{3}{|c|}{} & \multicolumn{3}{|c|}{} \\[-9pt]
\multicolumn{3}{|c|}{On $(\CC, 0)$} & \multicolumn{3}{|c|}{on $(\CW, 0)/\Gamma$} \\[3pt]
\hline 
normal & order  & isotropy & vector & differential & quadratic\\
form for & $\nu \in \ZZ$ \& & group & field & 1--form & differential \\
a germ $X$ & residue  & $\Gamma$ & $Y$ & $\omega_{Y}$ & 
$\omega_{Y}\otimes\omega_{Y}$ \\
& ${\tt r} \in \CC$&&&&\\
\hline
\hline
& & & & & \\[-8pt]
$\frac{1}{z^\nu} \del{}{z}$ & $- \nu \leq -1$ & $\ZZ_{\tt k}, $ 
& $\frac{1}{w^{(\nu+1)/ {\tt k} -1 }}\del{}{w}$ & $w^{(\nu+1)/ {\tt k} -1 } dw$  & $w^{2(\nu+1)/ {\tt k} -2 } dw^2$ \\
& & ${\tt k} \vert (\nu+1)$ & & & \\[4pt]
\hline
& & & & & \\[-8pt]
$\lambda z \del{}{z}$ & $\nu=1$ & $\CC^*  \triangleright \ZZ_{\tt k}$ & 
$\frac{\lambda w}{\tt k} \del{}{w}$ & $\frac{\tt k}{\lambda w} dw$ & 
$\frac{{\tt k}^2 }{ \lambda^2 w^2} dw^2$ \\
& ${\tt r}= \/ \lambda$ & & & &\\
\hline
& & & & & \\[-8pt]
$z^2 \del{}{z}$ & $\nu=2$ & $id$ & $w^2 \del{}{w}$ & $\frac{1}{w^2} dw$ & $\frac{1}{w^4} dw^2$  \\[4pt]
\hline
& & & & & \\[-8pt]
$z^\nu \del{}{z}$ & $\nu \geq 3$ & $\ZZ_{\tt k} ,$ &  
$w^{(\nu-1)/{\tt k} +1} \del{}{w}$ & $\frac{1}{w^{(\nu-1)/{\tt k} +1}} dw$ & $\frac{1}{w^{2(\nu - 1)/ {\tt k} +2 } } dw^2$ \\
& ${\tt r}=0$ & ${\tt k} \vert (\nu-1)$ & & & \\[2pt]
\hline
& & & & & \\[-8pt]
$\frac{z^\nu}{1 + \lambda z^{\nu-1}} \del{}{z}$ & $\nu \geq 3$ &  $id$ &
$\frac{w^{\nu } }{1 + \lambda w^{\nu -1}  } \del{}{w}$ & $\frac{1 + \lambda w^{\nu -1}  }{w^{\nu } } dw$ &
$\frac{ (1 + \lambda w^{\nu -1} )^2 }{ w^{2\nu } } dw^2$\\
& ${\tt r}= \lambda \neq 0$ &  & & & \\[2pt]
\hline
& & & & & \\[-8pt]
$\e^{z^d} \del{}{z}$
&  $\nu \geq 3$ & $\ZZ_{\tt k},$ &
$\e^{w^{d/ {\tt k} } } \del{}{w}$ & $\e^{-w^{d/ {\tt k} } }  dw$ & $\e^{-w^{2d/ {\tt k} } }  dw^2$ \\
& ${\tt r}=0$ & ${\tt k} \vert d$ & & &\\[2pt]
\hline

\end{tabular}
\end{center}
\label{Tabla}
\end{table}

\subsection{Description of 
$Y = proj_* X$, for $X \in \E(s,r,d)$ }
\

\noindent 
Recall that for $X \in \E(s,r,d)$; the rotation 
$\langle T_{\tt k} : z \mapsto \e^{2 \pi i /{\tt k}} z + {\tt b} \rangle$ 
is the generator of the isotropy group  $Aut(\CC)_X$, $C$ 
is the center of rotation of $T$ and
$proj :\CW_{z} \longrightarrow \CW_{z}/ \ZZ_{\tt k}
=\CW_{w}$.

\begin{proposition}\label{propCociente}
Let $X \in \E(s,r,d)$ having  $Aut(\CC)_X\cong\ZZ_{\tt k}$, 
${\tt k} \geq 2$, as isotropy group. 
The quotient vector field  \ $Y =  proj_* X$ \
has the following characteristics.  
 
\begin{enumerate}[label=\arabic*),leftmargin=*]
\item
$Y\in\E(s^\prime,r^\prime,d^\prime)$ has $s^\prime$ zeros, 
$r^\prime$ poles 
and an essential singularity of 1--order $d^\prime$ at $\infty$, where 

\noindent
$\bullet\ d^\prime=d/{\tt k}$,  

\noindent
$\bullet\ s^\prime=s/{\tt k}$, $r^\prime=\frac{r+1}{\tt k}-1$ when $C$ is a pole of $X$,

\noindent
$\bullet\ r^\prime=r/{\tt k}$, $s^\prime=\frac{s-1}{\tt k}+1$ when $C$ is a zero of $X$. 

\item 
The isotropy of $Y$ in $Aut(\CC)$ is trivial.
 
\item 
The phase portrait of $X$ is the pullback via 
$\{ z \mapsto \e^{2 \pi i / {\tt k} } z + {\tt b}\}$ 
of the phase portrait of $Y$.

\end{enumerate}
\end{proposition}

\begin{proof} 
Since $\Psi_{X}(z)=\int^{z} \omega_{X}$, 
the diagram \eqref{diagramaEquivariante} commutes and 
assertions (2) and (3) follow. 

Now, we compute the nature of the singularities of $Y$. 

If $d> 1$, then   
$\infty$ is an isolated essential singularity of $X$ having
$2 d$ entire sectors
(\S 5.3.1 pp. 151, figure 3 pp. 153 \cite{AlvarezMucino}). 
By  theorem (A) pp. 130, 
Corollary 10.1 pp. 216 in \cite{AlvarezMucino}, it follows that since
$proj$ is ${\tt k}$ to 1 around $\infty$ and since ${\tt k}\vert d$ then 
the phase portrait of
${proj}_*(X)$ 
has $2 d^\prime =d/ {\tt k}$ entire sectors at $\infty\in\CW_z$.

For the number $s^\prime$ of zeros and $r^\prime$ of poles of 
${proj}_*(X)$, 
recalling Theorem \ref{realization} we need to consider two cases: 
(${\tt k}\vert s$ and ${\tt k}\hspace{-4pt}\not\vert r$) and 
(${\tt k}\hspace{-4pt}\not\vert s$ and ${\tt k}\vert r$).

\smallskip 
\noindent
\textbf{Case (${\tt k}\vert s$ and ${\tt k}\hspace{-4pt}\not\vert r$): 
$C$ is a pole of $X$}.
Note that  
\begin{align*}
r &={\tt k} k_r + \nu \quad \text{ with } \quad k_r,\nu \in\NN\cup\{0\}, \ \ {\tt k}\hspace{-4pt}\not\vert \nu,
\ \ {\tt k}\vert(\nu+1) \\
s &={\tt k} k_s \quad \text{ with }\quad k_s\in\NN\cup\{0\}.
\end{align*} 
In this case the fundamental region, induced by $T_{\tt k}$, 
has exactly $k_r + \nu$ poles of $X$ 
($C$ being a pole of multiplicity $\nu$)
and $k_s$ zeros of $X$.
The phase portrait of $X$ has $2(\nu+1)$ hyperbolic sectors at $C$.

\noindent
On the other hand, 
${proj}_*(X)$ corresponds to a vector field $Y$
on $\CW/ Aut(\CC)_X$ and
a local condition at $proj(C)$ must be met:
$Y$ should have a pole of order $\nu^\prime$ hence $Y$ is required to have 
$2(\nu^{\prime}+1)$ hyperbolic sectors 
at $proj(C)$ hence $\frac{2(\nu+1)}{\tt k}=2(\nu^{\prime}+1)$ 
so $\nu^{\prime}=\frac{\nu+1}{\tt k}-1$.
In other words the local condition is equivalent to ${\tt k}\vert(\nu+1)$. 

\noindent
Thus ${proj}_*(X)\in\E(s^\prime, r^\prime, d^\prime)$ for 
$s^\prime=s/{\tt k}$, $d^\prime = d/ {\tt k}$ and $r^\prime=k_r + \nu^\prime$ where 
$\nu^\prime=\frac{\nu+1}{\tt k}-1$, so $r^\prime=\frac{r+1}{\tt k}-1$.

\smallskip
\noindent
\textbf{Case (${\tt k}\hspace{-4pt}\not\vert s$ and ${\tt k}\vert r$): $C$ is a zero of $X$}. 
In this case 
\begin{align*}
r &={\tt k} k_r \quad \text{ with }\quad k_r\in\NN\cup\{0\}, \\
s &={\tt k} k_s + \nu \quad \text{ with } \quad 
k_s,\nu \in\NN\cup\{0\}, \ \ {\tt k}\hspace{-4pt}\not\vert \nu,
\ \ {\tt k}\vert(\nu-1).
\end{align*}
The corresponding argument then yields that 
${proj}_*(X)\in\E(s^\prime, r^\prime, d^\prime)$, 
for
$r^\prime=r/{\tt k}$, $d^\prime = d/ {\tt k}$ and 
$s^\prime=\frac{s-1}{\tt k}+1$.
\end{proof}

See for instance Examples \ref{ejemploE023}, \ref{ejemploE032s} and Figures \ref{fig7ejemploscampos} (a), \ref{E03E05} (c) respectively.

\begin{remark}\label{stratification}
The map ${proj}_*$ is well defined on 

\centerline{
$\mathcal{U}_{\tt k}=
\{ X\in\E(s,r,d)\ \vert\ Aut(\CC)_{X}\cong\ZZ_{\tt k} \}$.}

\noindent 
Thus Proposition \ref{propCociente} 
provides a certain reducibility  property 
$$
\mathcal{U}_{\tt k} 
\longrightarrow
\E(s^\prime,r ^\prime,d/{\tt k})_{id},
\ \ \
X \longmapsto  proj_* X=Y.
$$
\end{remark}

\subsection{Rational vector fields}\label{familiasEspeciales}
By relaxing the condition that $d\geq1$, {\it i.e.} considering $d=0$, 
we then have the family

\centerline{
$\E(s,r,0)=\Big\{ X(z)
= \frac{Q(z)}{P(z)}\del{}{z}\ \vert\ Q,\ P\in\CC[z],\ \deg Q=s,\ \deg P=r \Big\}$,} 

\noindent
of rational vector fields on the sphere with $s$ 
zeros and $r$ poles on $\CC$.

\noindent
The main difference between the case $d=0$ and $d\geq 1$ is the dynamical behaviour of $\infty\in\CW$.
By Poincar\'e--Hopf theory, 
$X\in\E(s,r,0)$ has $\infty\in\CW$ as
\begin{enumerate}[label=\alph*)]
\item a \emph{regular point} when $2-s+r=0$, 
\item a \emph{zero of order $\mu$} when $\mu=2-s+r \geq  1$, and 
\item a \emph{pole of order $-\nu$} when $\nu=2-s+r \leq  -1$.
\end{enumerate} 
Obviously, as 
the following examples
show, generically for $X\in\E(s,r,0)$ the isotropy group $Aut(\CW)_{X}$ \emph{does not fix $\infty\in\CW$} (and hence strays from the present work).
For further examples and a classification of rational vector fields with finite isotropy on the Riemann sphere, see \cite{AlvarezFriasYee}.

\begin{example}\label{IsotropiaNoFijaInf}
1. Consider
\begin{equation}\label{diedricoN}
X(z)=\lambda\ \frac{z(z^n-1)}{z^n+1} \del{}{z}\in\E(n+1,n,0), 
\text{ for } n\geq 3.
\end{equation}
As shown in \cite{AlvarezFriasYee}, the isotropy group 
is a dihedral group
$Aut(\CW)_X\cong \mathbb{D}_{n}$.
In this case $\{z\mapsto -{1}/{z}\} \in Aut(\CW)_X$, hence $\infty\in\CW$ is 
not a fixed point of the isotropy group.
See Figures \ref{diedro-tetraedro} (A) and \ref{diedro-tetraedro} (B).
\begin{figure}[htbp]
\begin{center}
\includegraphics[width=.23\textwidth]{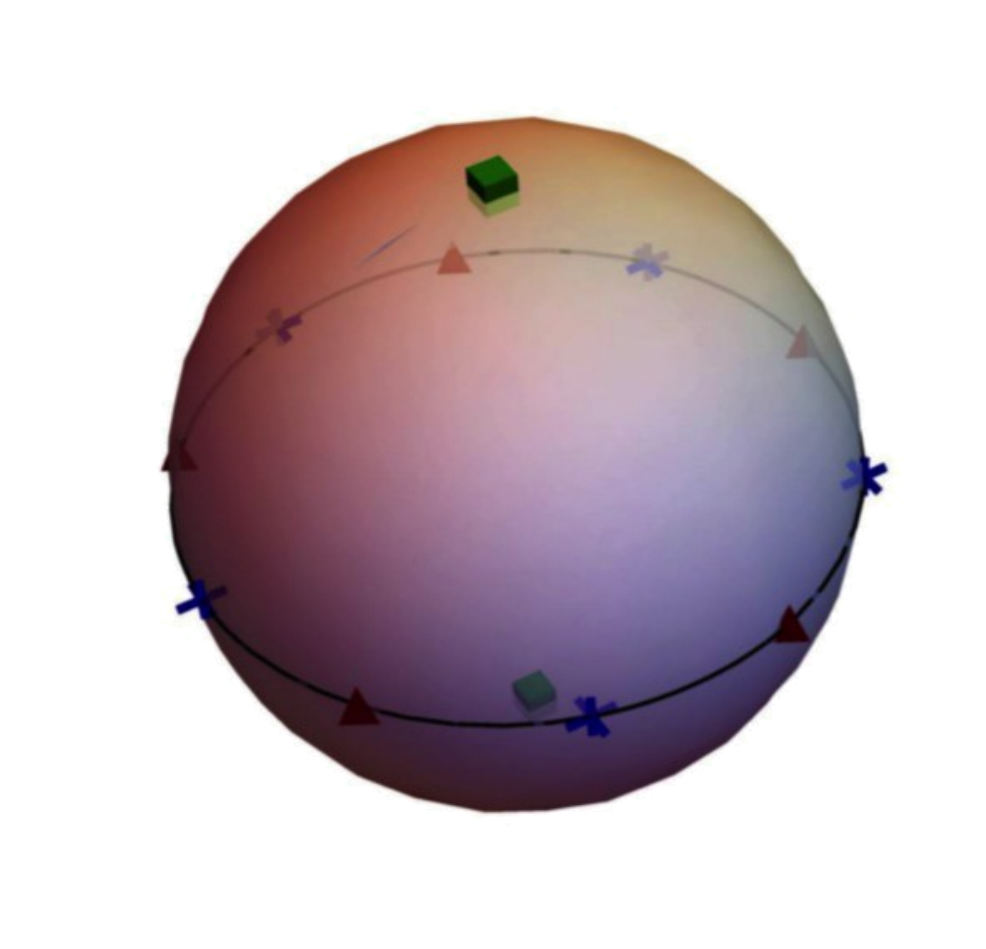}
\hspace {5pt}
\includegraphics[width=.23\textwidth]{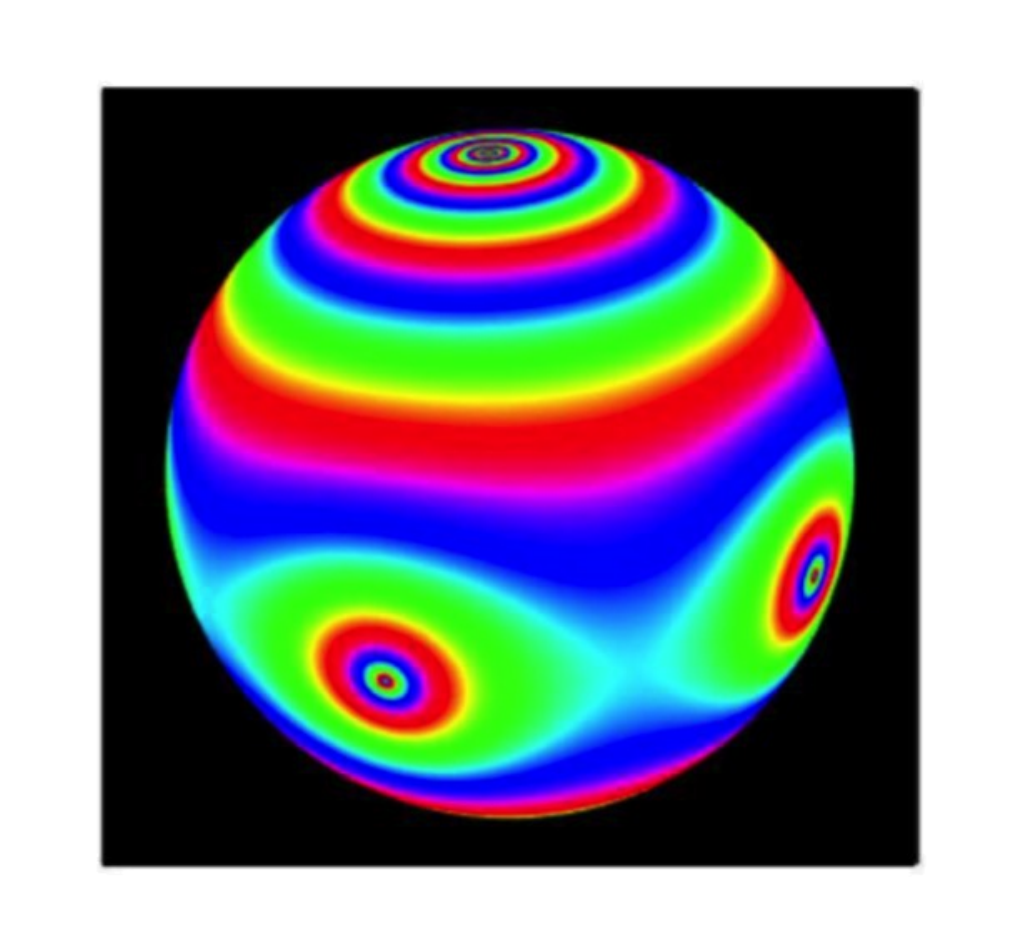}
\hspace {5pt}
\includegraphics[width=.23\textwidth]{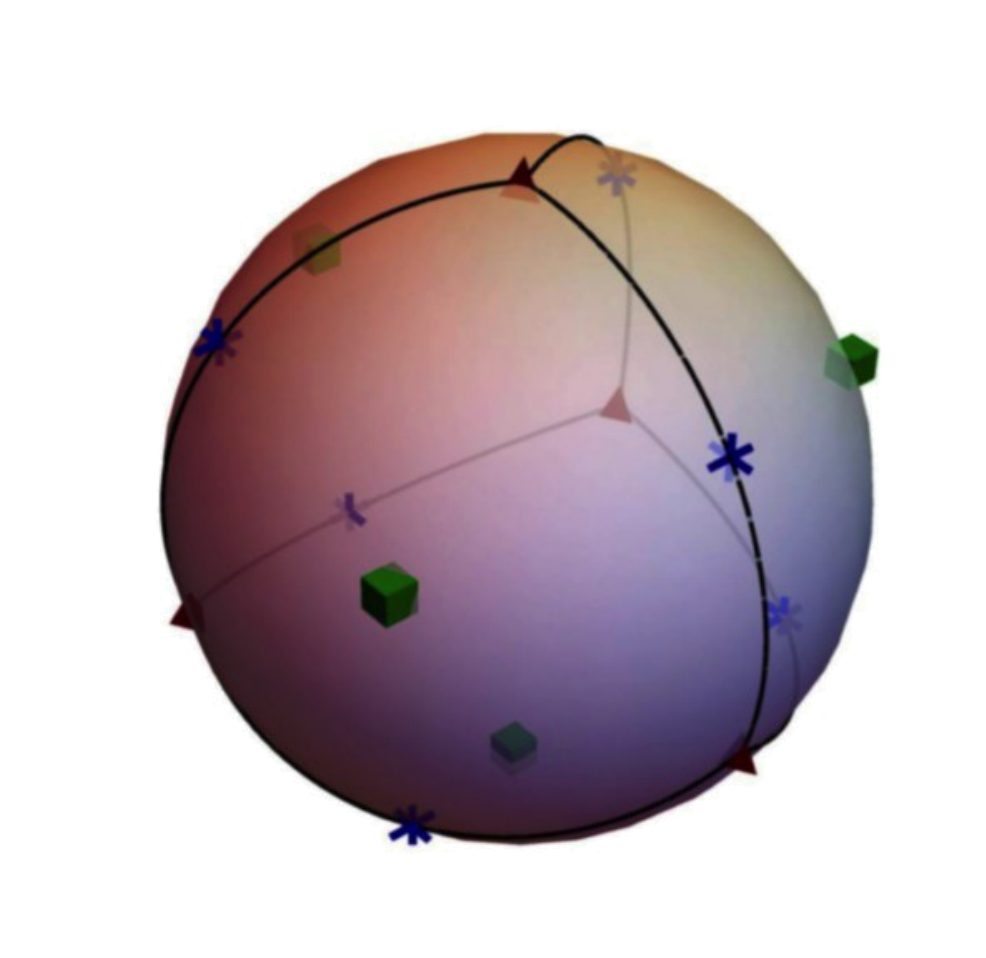}
\hspace {5pt}
\includegraphics[width=.23\textwidth]{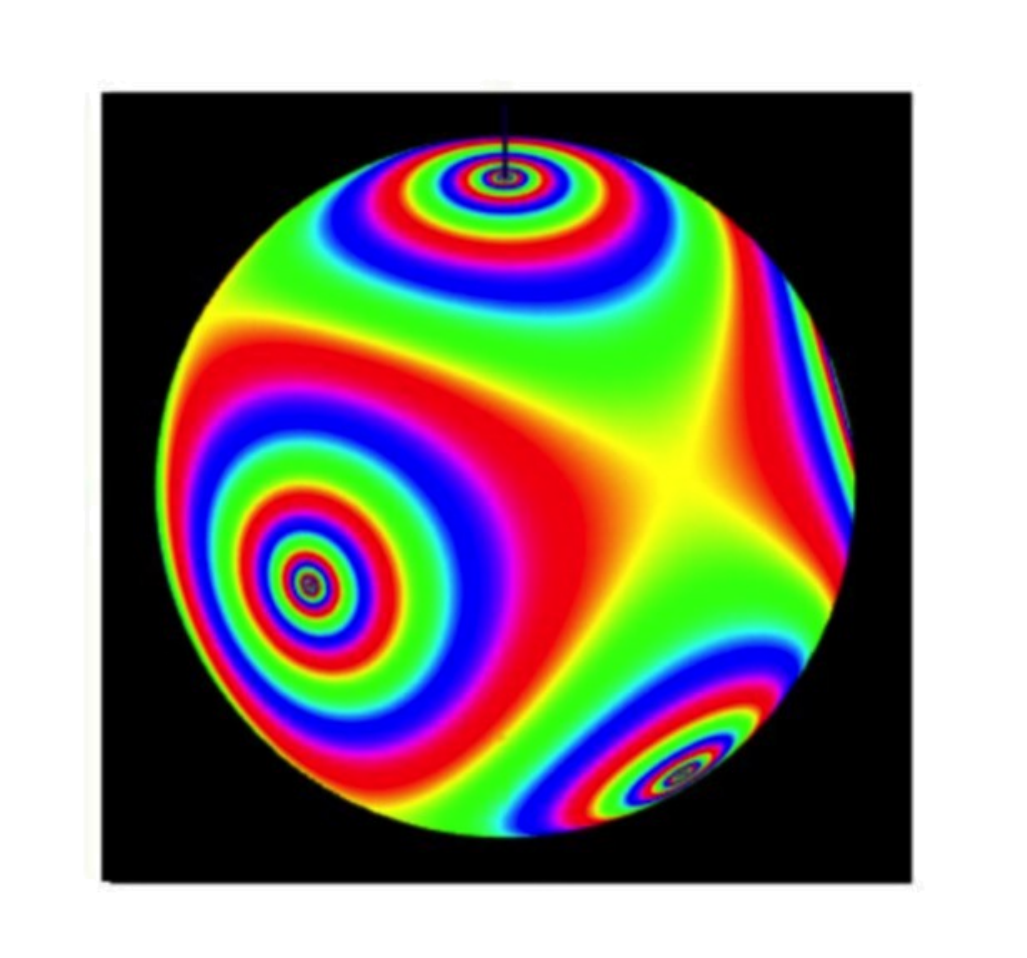}
\\
(A)\hspace{95pt} (B)\hspace{100pt} (C)\hspace{95pt} (D)
\caption{Phase portraits of Example \ref{IsotropiaNoFijaInf}. 
We have set $\lambda=-i$ so that the zeros of $X$ are centers. 
(A) and (C) represent the divisors of $X$: zeros appear as red pyramids, 
poles appear as blue crosses.
In (B) and (D) the corresponding phase portraits are visualized.
Borders of the strip flows correspond to streamlines of the field.
(A) and (B) correspond to \eqref{diedricoN} with $n=5$ which has isometry group 
isomorphic to $\mathbb{D}_{5}$. 
(C) and (D) correspond to \eqref{tetra} 
which has isometry group isomorphic to $A_{4}$.
}
\label{diedro-tetraedro}
\end{center}
\end{figure}

\noindent
From the perspective of Theorem \ref{RelacionCampoFuncion},
$\CW/ \mathbb{D}_n = \CW$ and
$$
proj: \CW \longrightarrow \CW, 
\ \ \ \ \ \
proj_* X(w) = n \lambda  \frac{w(w-1)}{w +1} \del{}{w} \doteq Y(w).
$$ 

\noindent
Moreover, a quick calculation involving partial fractions shows that the 
distinguished parameter

\centerline{$
\Psi_{X}(z)= \frac{2}{n}\log \left(1-z^n\right)-\log (z) $}

\noindent 
is multivalued 
and has $\mathbb{D}_{n}$--symmetry.

\smallskip
\noindent
2. Consider 
\begin{equation}\label{tetra}
X(z)=\lambda\ \frac{4 z^7+7 \sqrt{2} z^4-4 z} {4 z^6-20 \sqrt{2} z^3-4} \del{}{z}\in\E(7,6,0).
\end{equation}
In this case, as shown in \cite{AlvarezFriasYee}, the isotropy group $Aut(\CW)_X \cong \mathbb{A}_{4}$, 
the isometry group of the tetrahedron.
Note that $\infty\in\CW$ is a vertex of the corresponding tetrahedron and 
since the vertices are in the same orbit of 
$Aut(\CW)_X$, it follows that $\infty\in\CW$ is not a fixed point of the isotropy group.
See Figures \ref{diedro-tetraedro} (C) and \ref{diedro-tetraedro} (D).

\noindent
Similarly, from the perspective of Theorem \ref{RelacionCampoFuncion},
$\CW/A_4 = \CW$ and 
$$
proj: \CW \longrightarrow \CW
\ \ \ \ \
proj_* X(w) = 4 \lambda w \del{}{w} \doteq Y(w). 
$$
Once again, the distinguished parameter

\centerline{
$ \Psi_{X}(z)
=
-i \left(2 \tanh ^{-1}\left(\frac{4 \sqrt{2} z^3}{9}+\frac{7}{9}\right)
+\log (z)\right)
$}

\noindent 
is multivalued and has $A_4$--symmetry. 
\end{example}

\begin{remark}\label{PsiUnivaluada}
The above behaviour of $\Psi_{X}$ is worth noting:
$\Psi_{X}$ is a single valued function if and only if $\omega_X$ has zero residue 
on all its poles.
\end{remark}

The cases $s=d=0$ and $r=d=0$ are of special interest.


\subsubsection{The families $\E(0,r,0)$}
A particularly interesting case is $\E(0,r,0)$; 
the condition that $\infty\in\CW$ is a fixed point of $Aut(\CW)_{X}$ is 
automatically satisfied. 
In this case, there is a zero of multiplicity $r+2$ at $\infty\in\CW$,
and multi--saddles in $\CC$.

The family $\E(0,r,0)$ appears in
W. Kaplan \cite{Kaplan} and 
W. Boothby \cite{Boothby1}, \cite{Boothby2}.
On the other hand, M. Morse and J. Jenkins \cite{Morse-Jenkins}  
studied whether a foliation on the 
plane with multi--saddles as singularities can be recognized as the 
level curves of an harmonic function, 
see also R. Bott \cite{Bott}, 
\S 8, see also \cite{MR}.
So by 
using the dictionary, Proposition \ref{basic-correspondence}, we recognize

\centerline{$
X(z)=\frac{1}{P(z)}\del{}{z} 
\ \longleftrightarrow \   \Psi(z)=\int^{z} P(\zeta) d\zeta
$.}

As an immediate corollary of the Main Theorem we have:

\begin{corollary}
[Analytical and metric classification of $\E(0,r,0)$]\label{thmNormalFormsE0r0}
\hfill
\begin{enumerate}[label=\arabic*)]
\item The families
$\E(0,r,0)$ and $\E(0,r,0)_{id}$ coincide if and only if $r+1$ is prime.

\smallskip
\noindent
\hspace{-18pt}
For $r\geq 2$:

\item $\pi_1: \E(0, r,0)_{id} \longrightarrow \E(0,r,0)_{id}/{Aut(\CC)}$ 
is a holomorphic trivial principal bundle, 
\\
$\pi_2 \circ \pi_1: \E(0, r,0)_{id} 
\longrightarrow \E(0,r,0)_{id}/(Aut(\CC)\times \mathbb{S}^{1})$
is a real analytic trivial principal bundle.

\item 
\label{realizationE0r0}
If $X\in\E(0,r,0)\backslash\E(0,r,0)_{id}$ then there exists a rotation group 
$\Gamma\cong\ZZ_{\tt k}$ 
for ${\tt k}\in\mathscr{D}\backslash\{1\}$ and ${\tt k}\hspace{-4pt}\not\vert r$ 
that leaves invariant 
\begin{equation*}
X(z) = \frac{ \lambda }
{ (z-C)^{\nu} \prod\limits_{j=1}^{k_{r}} \prod\limits_{\ell=1}^{\tt k}
\Big[z-C - (R_{j}\e^{i\alpha_j})^{\ell/{\tt k} } \Big]} \del{}{z},
\end{equation*}
where $r={\tt k} k_{r}+\nu$,
$\{R_j\} \subset\RR^{+}$, $\{\alpha_j\} \subset\RR$ and $\nu\in\NN$.
\end{enumerate}
\end{corollary}
\hfill\qed

Furthermore the corresponding normal form is given by 
\eqref{seccionGlobalR2} with $s=d=0$,

\centerline{
$X(z)= \frac{1}{z^{r} + b_2 z^{r-2} + \ldots + b_r} \del{}{z}$.}

\begin{example}
\label{ejemplosimples}
1. Consider 

\centerline{
$X_{1}(z)=\frac{1}{z(z^{2}-1)}\del{}{z}\in\E(0,3,0)^{S},
\quad
X_{2}(z)=\frac{1}{z(z^{2}-1)(z^{2}+4)}\del{}{z}\in\E(0,5,0)^{S}.$
}

\noindent 
Both have isotropy group isomorphic to $\ZZ_{2}$, 
in agreement with Proposition \ref{casosimplepolofijo}, 
see Figure \ref{E03E05} (A), (B).

\noindent
2. Let 

\centerline{
$X(z)=\frac{\lambda}{z^{3}(z^{4}-1)^{2}(z^{4}-16)}\del{}{z} \in\E(0,15,0).$
}

\noindent
Considering the partition $r=15=3+(4+4)+4$, and since $4 | (15+1)$, then $Aut(\CC)_{X}\cong\ZZ_{4}$ as can 
readily be seen by 
checking with \eqref{completeaction}, see Figure \ref{E03E05} (C).

\noindent
3. Consider 

\centerline{
$X(z)=\frac{\lambda}{z^{2}(z^{3}-1)(z^{3}+8)^{2}}\del{}{z} \in\E(0,11,0).$
}

\noindent
From the partition $r=11=2+(3+3)+3$, and since $3 \vert (11+1)$, it follows that $Aut(\CC)_{X}\cong\ZZ_{3}$ as 
can readily be seen 
by checking with \eqref{completeaction}, see Figure \ref{E03E05} (D).
\\
Since $r=11=3+4+4$ and $4 | (11+1)$, then $Aut(\CC)_{X}\cong\ZZ_{4}$ is also possible:  

\centerline{
$X(z)=\frac{\lambda}{z^{3}(z^{4}-1)(z^{4}+16)}\del{}{z} \in\E(0,11,0)$
}

\noindent
realizes it, see Figure \ref{E03E05} (E).

\begin{figure}[htbp]
\begin{center}
\includegraphics[width=0.18\textwidth]{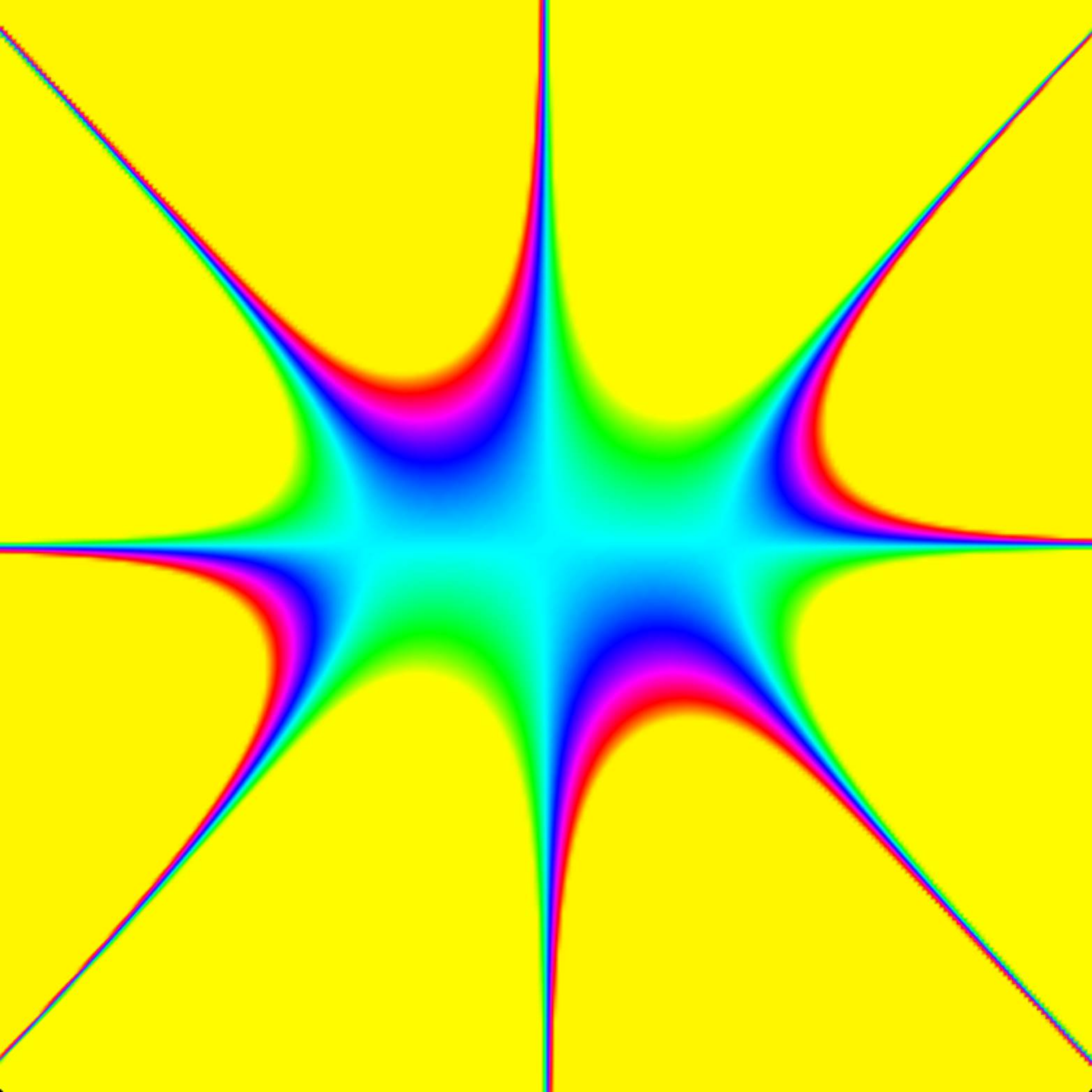}
\hskip 5pt
\includegraphics[width=0.18\textwidth]{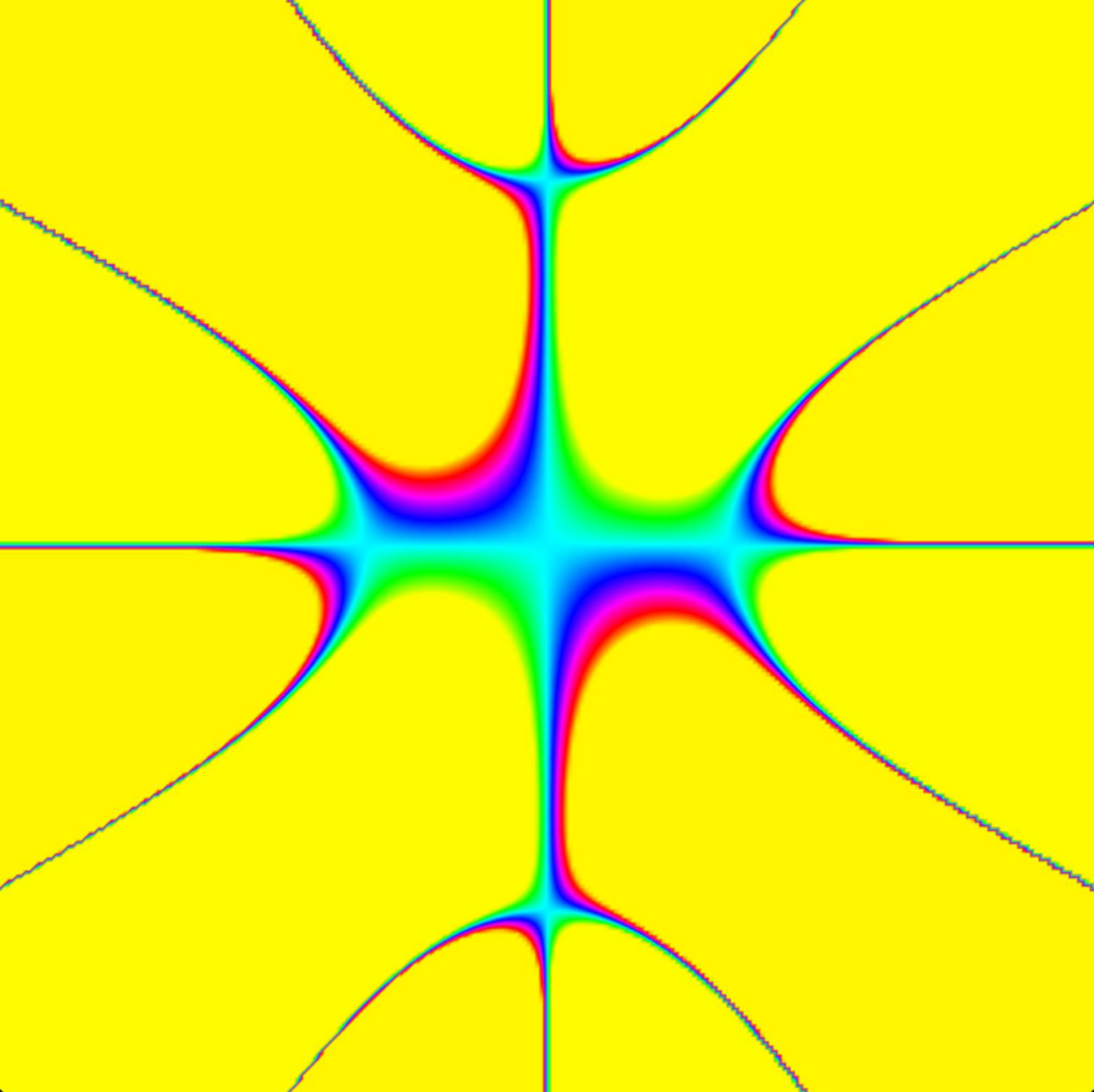}
\hskip 5 pt
\includegraphics[width=0.18\textwidth]{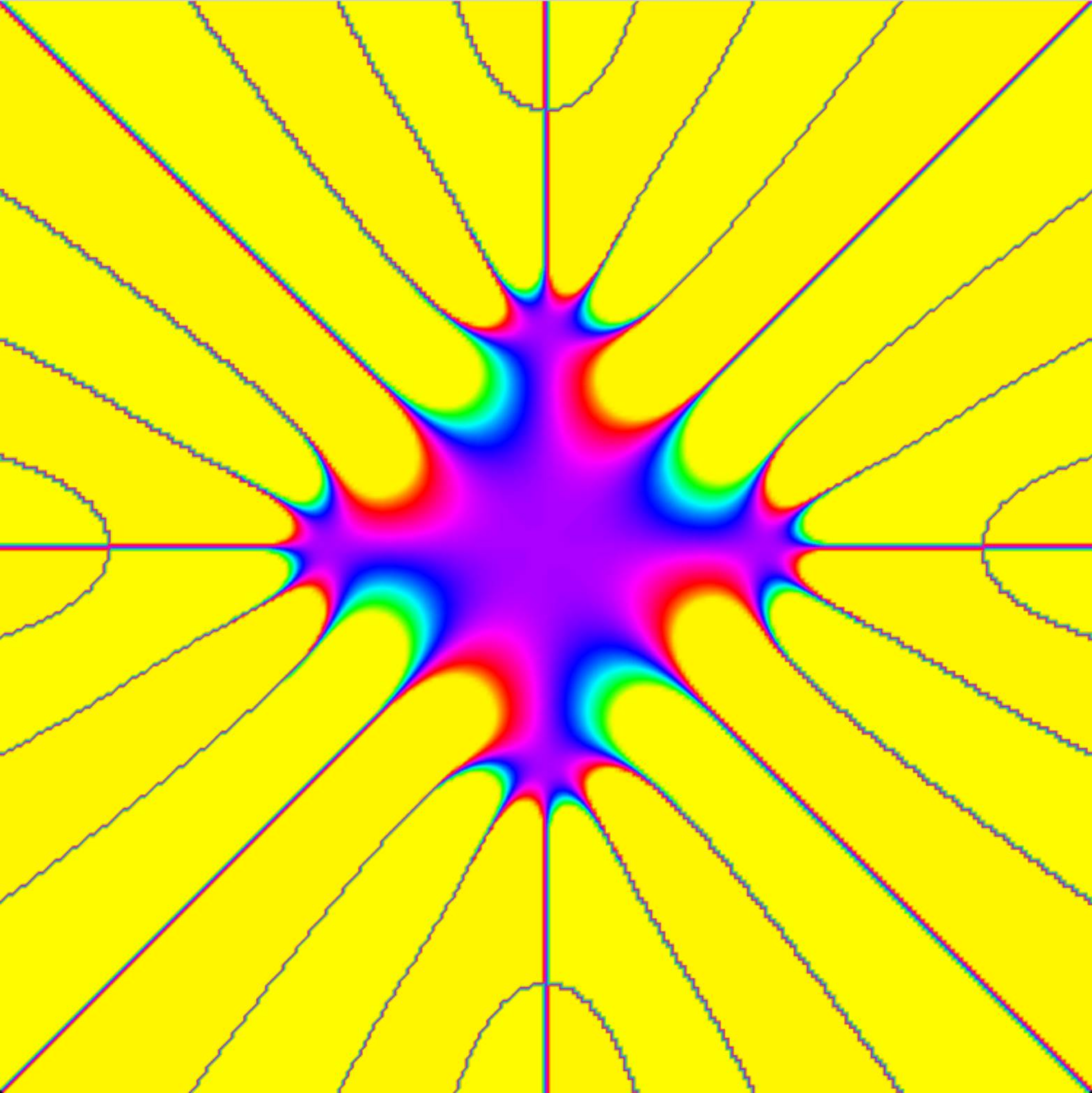}
\hskip 5 pt
\includegraphics[width=0.18\textwidth]{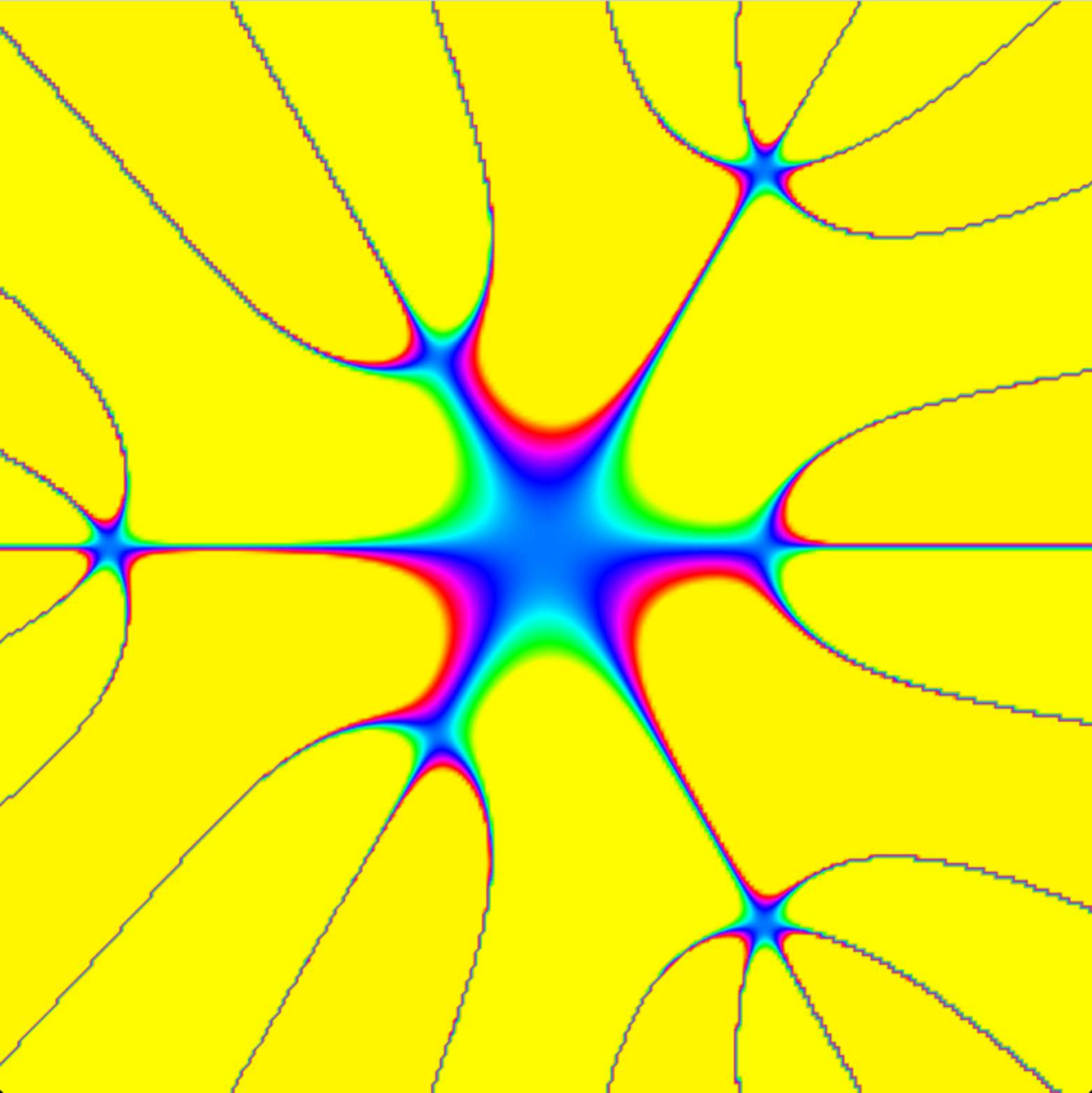}
\hskip 5pt
\includegraphics[width=0.18\textwidth]{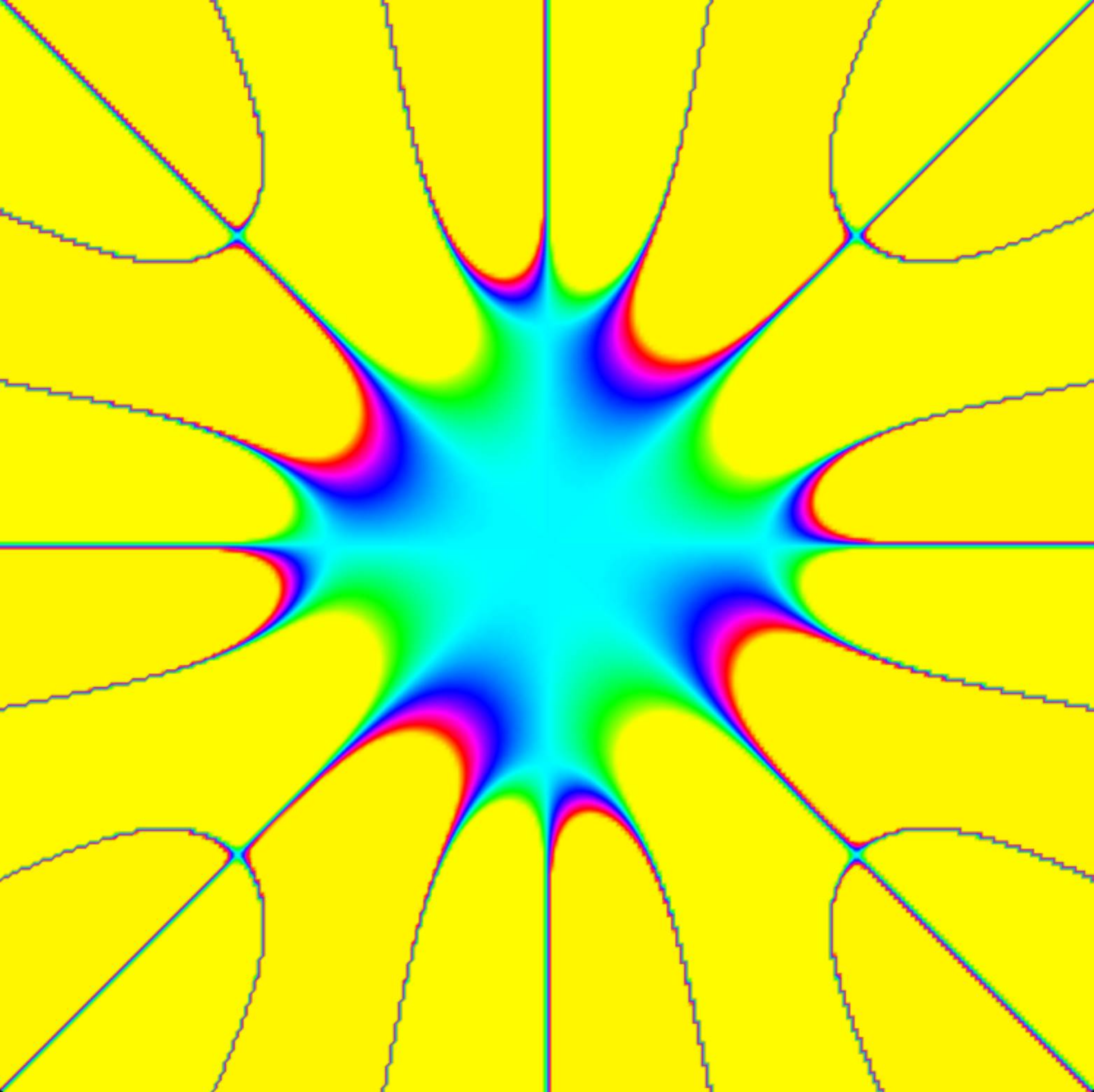}
\\
(A)\hskip 65pt (B) \hskip 65pt (C) \hskip 65pt (D) \hskip 65pt (E)
\caption{
Phase portraits of $\Re{X}$ in $\E(0,r,0)$ having non--trivial isotropy.   
Borders of the strip flows correspond to streamlines of the field.
(A) shows $X\in\E(0,3,0)^{S}$, (B) shows $X\in\E(0,5,0)^{S}$.
Both have isotropy group $Aut(\CC)_{X}\cong\ZZ_{2}$, see Example \ref{ejemplosimples}.1.
(C) corresponds to $X\in\E(0,15,0)$ with isotropy group isomorphic to $\ZZ_{4}$, see Example \ref{ejemplosimples}.2.
(D) and (E) correspond to $X\in\E(0,11,0)$ with (D) having isotropy group isomorphic to $\ZZ_{3}$ and 
(E) having isotropy group isomorphic to $\ZZ_{4}$, 
see Example \ref{ejemplosimples}.3.}
\label{E03E05}
\end{center}
\end{figure}
\end{example}

\subsubsection{The families $\E(s,0,0)$}

For the case $\E(s,0,0)$; the condition that 
$\infty\in\CW$ is a fixed point of $Aut(\CW)_{X}$ is automatically 
satisfied for $s\geq 3$: 
$X$ has a pole of order $2-s$ at $\infty\in\CW$.
Dynamically this corresponds to the case of 
singularities consisting of centers, sources, sinks and flowers 
on $\CC$ and a multi--saddle at $\infty$.

The polynomial vector fields $X \in \E(s,0, 0)$
have been studied by 
A. Douady \emph{et al.}
\cite{Douady-Estrada-Sentenac}, 
B. Branner \emph{et al.} \cite{Branner-Dias}, 
M.--E. Fr\'ias--Armenta \emph{et al.} 
\cite{Frias-Mucino}, 
C. Rousseau \cite{Rousseau}
amongst others.

Once again by the Main Theorem we have.
\begin{corollary}
[Analytical and metric classification of $\E(s,0,0)$]\label{thmNormalFormsEs00}
\hfill
\begin{enumerate}[label=\arabic*)]
\item The families
$\E(s,0,0)$ and $\E(s,0,0)_{id}$ coincide if and only if $s-1$ is prime.

\smallskip
\noindent
\hspace{-18pt}
For $s\geq 3$:

\item $\pi_1: \E(s, 0,0)_{id} \longrightarrow \E(s,0,0)_{id}/{Aut(\CC)}$ 
is a holomorphic trivial principal bundle, 
\\
$\pi_2 \circ \pi_1: \E(s, 0,0)_{id} 
\longrightarrow \E(0,r,0)_{id}/(Aut(\CC)\times \mathbb{S}^{1})$
is a real analytic trivial principal bundle. 

\item
\label{realizationEs00}
If $X\in\E(s,0,0)\backslash\E(s,0,0)_{id}$ 
then there exists a rotation group 
$\Gamma\cong\ZZ_{\tt k}$ 
for ${\tt k}\in\mathscr{D}\backslash\{1\}$ and ${\tt k}\hspace{-4pt}\not\vert s$ 
that leaves invariant $X$.
Furthermore
\begin{equation*}
X(z) = \lambda\  (z-C)^{\nu} \prod\limits_{j=1}^{k_{s}} \prod\limits_{\ell=1}^{\tt k} 
\Big[z-C - (r_{j}\e^{i\theta_j})^{\ell/{\tt k}} \Big] 
\end{equation*}
where $s={\tt k} k_{s}+\nu$,
$\{r_j\} \subset\RR^{+}$, $\{\theta_j\} \subset\RR$ and $\nu\in\NN$ such that $k\vert(\nu-1)$.
\hfill
\qed
\end{enumerate}
\end{corollary}

The corresponding normal form is given by \eqref{seccionGlobalS2} with $r=d=0$, $s\geq3$,
is

\centerline{
$X(z)= (z^{s} + a_2 z^{s-2} + \ldots + a_s) \del{}{z}$.}

\begin{example}\label{ejemploE700}
As an example consider $\E(7,0,0)$, note that $\mathscr{D}=\{1,3,6\}$.
The vector field

\centerline{
$X(z)= z^4 (z^3-1)\del{}{z}$
}

\noindent 
has $Aut(\CC)_{X}\cong\ZZ_{3}$. 
In this case, there is a saddle at $\infty\in\CW_{z}$ with 12 hyperbolic sectors (corresponding to a pole of $X$ of multiplicity $5=7-2$). 
See Figure \ref{figE700} for the phase portrait in the vicinity of the origin.
\begin{figure}[htbp]
\begin{center}
\includegraphics[width=0.3\textwidth]{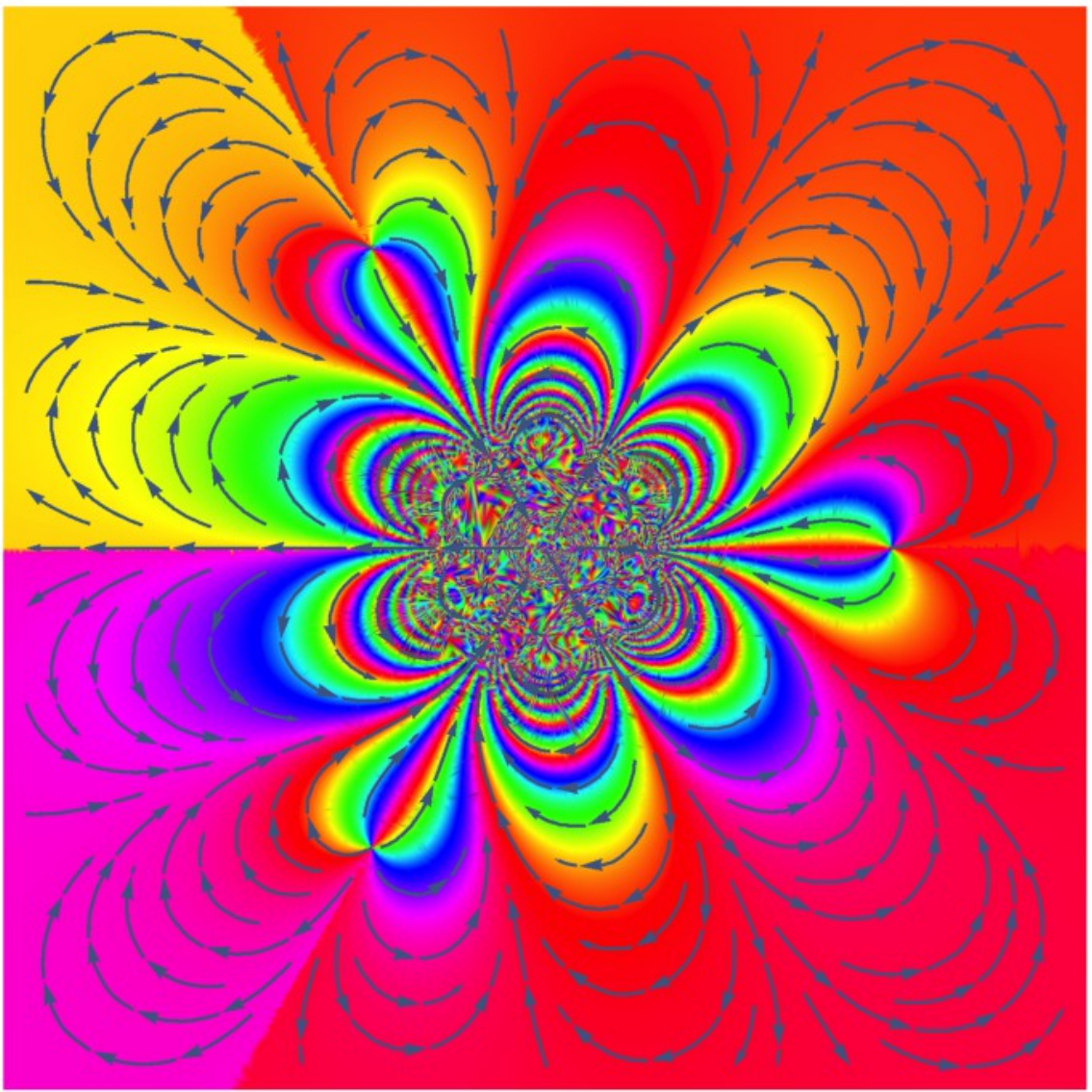}
\caption{Phase portrait of 
$\Re{X}$ for $X(z)=z^4 (z^3-1)\del{}{z}$ in $\E(7,0,0)$, 
with isotropy group
$Aut(\CC)_{X}=\ZZ_{3}$, 
see Example \ref{ejemploE700}.
Borders of the strip flows correspond to streamlines of the field.}
\label{figE700}
\end{center}
\end{figure}

\noindent
The distinguished parameter $\Psi_{X}$ has $\ZZ_{3}$--symmetry and is once 
again multivalued
\begin{equation*}
\Psi_{X}(z)=\frac{1}{3 z^3}+\frac{1}{3} \log \left(1-z^3\right)-\log (z).
\end{equation*}
\end{example}

\subsection{Doubly periodic vector fields}
\label{hipereliptica}
Let $w_1,w_2\in\CC$ determine the \emph{period lattice} 
$\Lambda=\{ m w_1+n w_2 \ \vert\ m,n\in\ZZ\}$, 
and hence the torus $\mathbb{T}=\CC/\Lambda$.
We may then consider the Weirstrass $\wp$--function  
$$
\wp(z)=\wp(z;w_1,w_2)=\frac{1}{z^2} 
+ \sum_{n^2+m^2\neq 0} \left(\frac{1}{(z+m w_1+n w_2)^2} 
-\frac{1}{(m w_1+n w_2)^2}\right),
$$
and its derivative $\wp'(z)$.
As is well known, letting $x=\wp(z)$, $y=\wp'(z)$, 
 $g_{2}$ and $g_{3}$ the Weirstrass invariants, 
the torus $\mathbb{T}=\CC/\Lambda$ can also be expressed as
\begin{equation}\label{simhiper}
\mathbb{T}\backslash[0]=\{(x,y)\ \vert\ y^2=4x^3 -g_2 \, x - g_3\}\subset\CC^2,
\end{equation} 
where $[0]$ corresponds to the 
class of $z=0$ in $\Lambda$, 
see \cite{Ahlfors} pp. 272. 

\noindent
Diagram \eqref{diagramaRX} with 
$M=\mathbb{T}$ and $X(z)=\frac{1}{\wp^\prime (z)} \del{}{z}$ is
\begin{center}
\begin{picture}(180,70)(0,20)


\put(-47,75){$\Big(\mathbb{T} , X(z)=\frac{1}{\wp^\prime (z)} \del{}{z} \Big)$}

\put(110,75){$\big(\R_X,\pi^*_{X,2}(\del{}{t})\big)$}

\put(108,78){\vector(-1,0){60}}
\put(65,85){$\pi_{X,1}$}

\put(133,65){\vector(0,-1){30}}
\put(138,47){$ \pi_{X,2} $}

\put(38,65){\vector(2,-1){73}}
\put(40,39){$ \Psi_X = \wp $}

\put(115,20){$\big(\CW_t,\del{}{t}\big), $}

\end{picture}
\end{center}
with $\R_X=\Big\{ \big(z,\wp(z)\big)\Big\}\subset\mathbb{T} \times \CW_t$.
Since $\wp(z)$ has a second--order pole at $[0]$,
it follows that $X(z)=\wp^*(\del{}{t})$ has zeros of order three at $[0]$ 
and three simple poles at the classes of the midpoints of the period lattice.
Moreover because of \eqref{simhiper} the three poles are 
$\{ (x,0)\ \vert\ 4x^3 -g_2 \, x - g_3=0 \} \subset \mathbb{T}$, 
see Figure \ref{campoWeirstrass} for a particular case. 

\begin{figure}[htbp]
\begin{center}
\includegraphics[width=\textwidth]{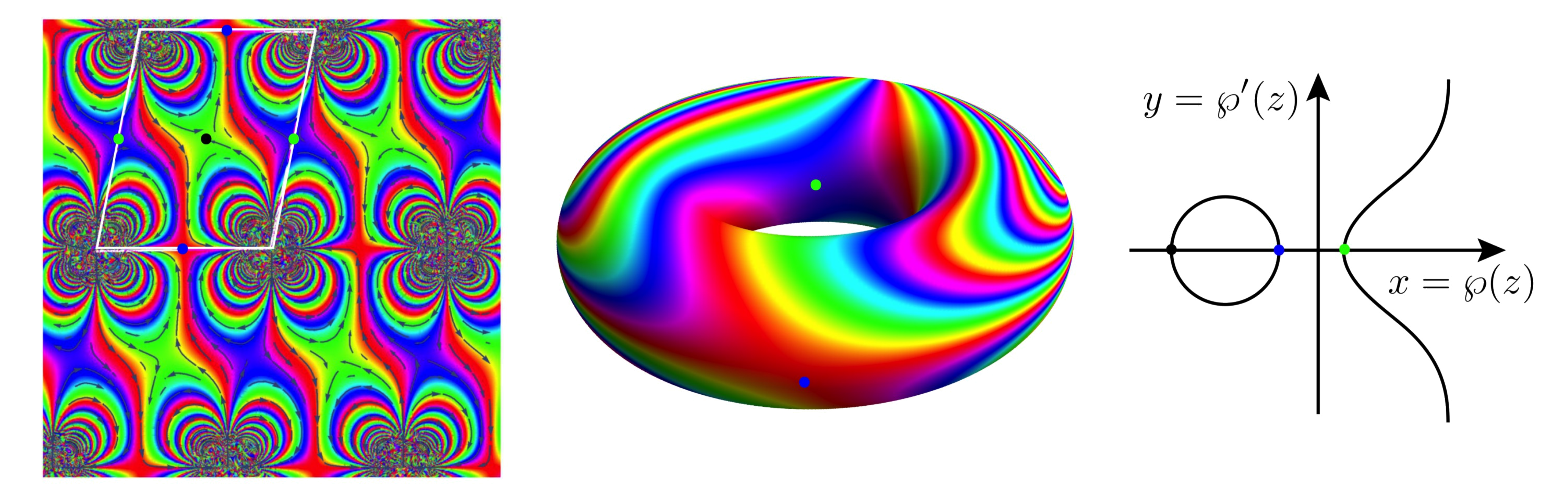}
\\
Fundamental domain of $\Lambda$ in $\CC$
\hspace{55pt}
$\mathbb{T}=\CC/\Lambda$
\hspace{50pt}
$\{ y^2 = 4x^3 -g_2 \, x - g_3 \}\subset\CC^2$
\caption{Representaton of $X(z)=\frac{1}{\wp'(z)}\del{}{z}$ as in Example \ref{hipereliptica} 
with $w_1=1$, $w_2= \frac{1}{4}+\frac{5}{4}i$.
On the top figure we have sketched the fundamental domain of $\CC / \Lambda$ 
together with the phase portrait of $\Re{X}$. 
In the bottom left a representation of the torus $\mathbb{T}$ with the phase portrait of $\Re{X}$.
On the bottom right we have sketched the plane algebraic curve of \eqref{simhiper}.
In all the models for $\mathbb{T}=\CC/\Lambda$ one can observe 
three simple poles and a zero of multiplicity three.}
\label{campoWeirstrass}
\end{center}
\end{figure}

From the above we can now recognize two examples of symmetries $\Gamma$.

\smallskip
\noindent
1. Recaling that $\Gamma=\ZZ_2$ acts on the torus (as a plane algebraic curve), 
having as generator the hyperelliptic 
symmetry 
$$
\mathbb{T}\longrightarrow \mathbb{T}, \quad (x,y)\longmapsto(x,-y).
$$  
It follows that in Diagram \eqref{diagramaEquivariante} $proj_{*}=\wp$ so

\centerline{
$
\Big(\mathbb{T} , X(z)=\frac{1}{\wp^\prime (z)} \del{}{z} \Big)  
\xrightarrow{\ proj_{*}  = \wp \ }
\Big(\CW =\mathbb{T}/\mathbb{Z}_2,  Y(t)=\del{}{t} \Big)  
\xrightarrow{\ \Psi_Y= Id \ }
\big(\CW_t, \del{}{t} \big).
$
}

\smallskip
\noindent
2. As a second example, let $M$ be a 
(branched) 
topological cover of $\mathbb{T}$
which inherits the conformal structure from $\mathbb{T}$, then 
the covering group $\Gamma$ is recognized as a subgroup of
the automorphism group of the Riemann surface $M$.
Letting $Y(z)=\frac{1}{\wp^\prime (z)} \del{}{z}$ on $\mathbb{T} = M/ \Gamma$, 
clearly $X=proj^* Y$. 
Then in Diagram \eqref{diagramaEquivariante} we can recognize

\centerline{
$
(M, X) 
\xrightarrow{\ proj_{*} \ }
\Big(\mathbb{T} = M/ \Gamma, Y(z)=\frac{1}{\wp^\prime (z)} \del{}{z} \Big)  
\xrightarrow{\ \ \ \Psi_Y =\wp \ \ \ }
\big(\CW_t, \del{}{t} \big),
$}

\noindent
where $\R_Y=\Big\{ \big(z,\wp(z)\big)\Big\}\subset\mathbb{T} \times \CW_t$.
%




\end{document}